\documentclass[envcountsame,envcountsect]{llncs}

\usepackage{amsmath,amssymb,enumerate,xspace,bbm,rotating}
\usepackage{gastex,xcolor}

\newcommand{\dA}{{\mathbb A}}

\newcommand{\dB}{{\mathbb B}}
\newcommand{\dC}{{\mathbb C}}
\newcommand{\dom}{\mathrm{dom}}
\newcommand{\Wgamma}{W_K}
\newcommand{\Aut}{\mathrm{Aut}}
\newcommand{\CWP}{\mathrm{CWP}}
\newcommand{\Emc}{{\mathcal{E}}}
\newcommand{\ran}{\mathrm{ran}}
\newcommand{\Red}{\mathrm{Red}}
\newcommand{\UCWP}{\mathrm{UCWP}}
\newcommand{\RUCWP}{\mathrm{RUCWP}}
\newcommand{\RCWP}{\mathrm{RCWP}}

\newcommand{\eval}{{\mathrm{val}}}
\newcommand{\Zmc}{{\mathcal{Z}}}

\begin{document}

\title{Compressed word problems in HNN-extensions and amalgamated products}
\author{Niko Haubold \and Markus Lohrey
\institute{Institut f\"ur Informatik, Universit\"at Leipzig\\
\email{\{haubold,lohrey\}@informatik.uni-leipzig.de}}}

\maketitle

\begin{abstract}
It is shown that the compressed word problem for 
an HNN-extension $\langle H, t \mid t^{-1} a t = \varphi(a) (a \in A) \rangle$ 
with $A$ finite is polynomial time Turing-reducible to the compressed
word problem for the base group $H$. An analogous result for 
amalgamated free products is shown as well.
\end{abstract}

\section{Introduction}

Since it was introduced by Dehn in 1910 \cite{Dehn10}, the {\em word problem} for
groups has emerged to a fundamental computational problem linking group theory, topology,
mathematical logic, and computer science. The word problem for a finitely
generated group $G$ asks, whether a given word over the generators of $G$
represents the identity of $G$, see Section~\ref{S groups} for more details.
Dehn proved the decidability of the word
problem for surface groups. On the other hand, 50 years after the 
appearance of Dehn's work, Novikov \cite{Nov58} and independently 
Boone \cite{Boo59} proved the existence of a finitely presented group
with undecidable word problem. However, many natural classes of groups with 
decidable word problem are known, as for instance finitely generated 
linear groups, automatic groups and one-relator groups. 
With the rise of computational complexity theory, also
the complexity of the word problem became an active research area. This
development has gained further attention by potential applications 
of combinatorial group theory for secure cryptographic systems
\cite{MyShUs08}.

In order to prove upper bounds on the complexity of the word problem
for a certain group $G$, a ``compressed'' variant of the word problem for $G$
was introduced in \cite{Loh06siam,LoSchl07,Schl06}. 
In the {\em compressed word problem} for $G$, 
the input word over the generators is not given
explicitly but succinctly via a so called {\em straight-line program} (SLP for
short). This is a context free grammar that generates exactly one word,
see Section~\ref{S SLP}.
Since the length of this word may grow exponentially with the size (number of 
productions) of the SLP, SLPs can be seen indeed as a succinct string
representation. SLPs turned out to be a very
flexible compressed representation of strings, which are well suited for
studying algorithms for compressed data, see
e.g. \cite{BeChRa08,GaKaPlRy96,Lif07,Loh06siam,MiShTa97,Pla94,PlaRy98}.  
In  \cite{LoSchl07,Schl06} it was shown that the word problem for the automorphism 
group $\Aut(G)$ of $G$ can be reduced in polynomial time to the 
{\em compressed} word problem for $G$. In \cite{Schl06}, it was shown
that the compressed word problem for a finitely generated free group $F$ 
can be solved in polynomial time. Hence, the word problem for 
$\Aut(F)$ turned out to be solvable in polynomial time \cite{Schl06}, which solved
an open problem from \cite{KMSS03}. Generalizations of this result can 
be found in \cite{LoSchl07}.

In this paper, we prove a transfer theorem for the compressed word problem
of {\em HNN-extensions} \cite{HiNeNe49}. For a base group $H$ with two isomorphic subgroups 
$A$ and $B$ and an isomorphism $\varphi : A \to B$, the corresponding
HNN-extension is the group
\begin{equation} \label{HNN}
G=\langle H,t \mid  t^{-1} a t = \varphi(a) \, (a \in A) \rangle.
\end{equation}
Intuitively, it is obtained by adding to $H$ a new generator $t$ 
(the {\em stable letter}) 
in such a way that conjugation of $A$ by $t$ realizes 
$\varphi$. The subgroups $A$ and $B$ are also 
called the {\em associated subgroups}.
A related operation is that of the {\em amalgamated free product} of two 
groups $H_1$ and $H_2$ with isomorphic subgroups 
$A_1 \leq H_1$, $A_2 \leq H_2$ and an isomorphism $\varphi : A_1 \to A_2$.
The corresponding amalgamated free product is the group
$$
G=\langle H_1*H_2 \mid a=\varphi(a) \, (a\in A_1)\rangle .
$$
Intuitively, it results from the free product $H_1*H_2$ by
identifying every element $a\in A_1$ with $\varphi(a) \in A_2$. The subgroups
$A_1$ and $A_2$ are also called the {\em amalgamated} (or {\em identified}) 
subgroups.

HNN-extensions were introduced by Higman, Neumann, and Neumann in 1949
\cite{HiNeNe49}. They proved that $H$ embeds into the group $G$ from
(\ref{HNN}). Modern proofs of the above mentioned Novikov-Boone theorem
use HNN-extensions as the main tool for constructing finitely presented
groups with an undecidable word problem \cite{LySch77}. In particular,
arbitrary HNN-extensions do not preserve good algorithmic properties
of groups like decidability of the word problem.
In this paper, we restrict to HNN-extensions (resp.
amalgamated products) with finite
associated (resp. identified) subgroups, which is an important 
subcase. Stallings proved \cite{Stal71}
that a group has more than one end if and only if it is either an HNN-extension
with finite associated subgroups or an amalgamated free product with finite
identified subgroups. Moreover, a group is virtually-free (i.e., has a free
subgroup of finite index) if and only if it can be built up from finite groups
using amalgamated products with finite identified subgroups and HNN-extensions
with finite associated subgroups \cite{DiDu89}.

It is not hard to see that the word problem for an 
HNN-extension (\ref{HNN})
with $A$ finite can be reduced in polynomial time 
to the word problem of the base group $H$. The
main result of this paper extends this transfer theorem to 
the compressed setting: the compressed word 
problem for (\ref{HNN}) with $A$ finite can be reduced in polynomial time 
to the compressed word problem for $H$.
In fact, we prove a slightly more general result, which deals
with HNN-extensions with several stable letters $t_1, \ldots, t_n$,
where the number $n$ is part of the input.
For each stable letter $t_i$ the input contains 
a {\em partial} isomorphism $\varphi_i$ from the fixed finite subgroup $A \leq H$ to 
the fixed finite subgroup $B \leq H$ and we consider the multiple HNN-extension
$$
G = \langle H, t_1, \ldots, t_n \mid t_i^{-1} a t_i = \varphi_i(a)\ (1 \leq i \leq n, a \in \dom(\varphi_i)) \rangle.
$$
Our polynomial time reduction consists of a sequence of polynomial time
reductions. In a first step (Section~\ref{S reduced}), we reduce the
compressed word problem for $G$ to the same problem for {\em reduced sequences}.
These are strings (over the generators of $H$ and the symbols $t_1, t_1^{-1},
\ldots, t_n, t_n^{-1}$)
that do not contain a substring of the form $t_i^{-1} w t_i$ (resp. $t_i w t_i^{-1}$), 
where the string $w$ represents a group element from the domain (resp. range)
of $\varphi_i$. In a second step (Section~\ref{S constant number})
we reduce the number $n$ of stable letter to a constant $\delta$, which only 
depends on the size of the fixed subgroup $A$.
The main step of the paper reduces the compressed word problem for 
reduced sequences over an HNN-extension with $\delta$ many stable 
letters (and associated partial isomorphisms from $A$ to $B$)
into two simpler problems: (i) the same problem but with only $\delta-1$
many stable letters and (ii) the same problem (with at most $\delta$ many
stable letters) but with associated subgroups that are strictly smaller
than $A$. 
By iterating this procedure, we arrive after a constant number
of iterations (where each iteration is a polynomial time reduction)
at a compressed word problem for which we directly know the existence
of a polynomial time reduction to the compressed word problem
for the base group $H$. Since the composition of a constant
number of polynomial time reductions is again a polynomial time
reduction, our main result follows.

The main reduction step in our algorithm uses techniques similar to those
from \cite{LohSen06icalp}, where a transfer theorem for solving equations
over HNN-extensions with finite associated subgroups was shown.

{}From the close relationship of HNN-extensions with amalgamated free products, a
polynomial time reduction from the compressed problem for an amalgamated free
product $\langle H_1*H_2 \mid a=\varphi(a) \, (a\in A_1)\rangle$ (with $A_1$ finite)
to the compressed word problems of $H_1$ and $H_2$ 
is deduced in the final Section~\ref{amalprod}.

\section{Preliminaries}\label{S prel}

Let $\Sigma$ be a finite alphabet. The empty word is denoted by $\varepsilon$.
With $\Sigma^+ = \Sigma^*\setminus\{\varepsilon\}$ 
we denote the set of non-empty words over $\Sigma$.  
For a word $w=a_1\cdots a_n$ let $|w|=n$,
$\text{alph}(w)=\{a_1,\dots,a_n\}$, 
and $w[i:j]=a_i\cdots a_j$ for $1\leq i \leq j \leq n$. 
Moreover, let $w[i:] = w[i:n]$ and $w[:i] = w[1:i]$.

\subsection{Groups and the word problem} \label{S groups}

For background in combinatorial group theory see \cite{LySch77}.
For a group $G$ and two elements $x,y \in G$ we denote with
$x^y = y^{-1} x y$ the conjugation of $x$ by $y$.
Let $G$ be a \textit{finitely generated group} and let $\Sigma$ be a finite \textit{group
generating set} for $G$. Hence, $\Sigma^{\pm 1} = \Sigma\cup \Sigma^{-1}$ 
is a finite \textit{monoid generating set} for $G$ and there exists a canonical monoid homomorphism
$h:(\Sigma^{\pm 1})^*\rightarrow G$, which maps a word $w \in (\Sigma^{\pm
1})^*$ to the group element represented by $w$. For $u,v\in (\Sigma^{\pm
1})^*$ we will also say that $u=v$  in $G$ in case $h(u)=h(v)$.
The \textit{word problem for $G$ with respect to $\Sigma$} is the following
decision problem:

\medskip

\noindent
INPUT: A word $w\in(\Sigma^{\pm 1})^*$.

\smallskip

\noindent
QUESTION: $w=1$ in $G$?

\medskip

\noindent
It is well known that if $\Gamma$ is another finite
generating set for $G$, then the word problem for $G$ with respect to $\Sigma$
is logspace many-one reducible to the word problem for $G$ with respect to $\Gamma$.
This justifies one to speak just of the word problem for the group $G$.

The \textit{free group $F(\Sigma)$} generated by $\Sigma$ can be defined as the
quotient monoid
$$
F(\Sigma)= (\Sigma^{\pm 1})^*/\{aa^{-1}
=\varepsilon \mid a\in \Sigma^{\pm1}\}.
$$
A {\em group presentation} is a pair $(\Sigma, R)$, where $\Sigma$ is an
alphabet of symbols and $R$ is a set of {\em relations} of the form $u=v$, where
$u,v \in (\Sigma^{\pm 1})^*$. The group defined by this presentation is
denoted by $\langle \Sigma \mid R \rangle$. It 
is defined as the quotient $F(\Sigma)/N(R)$, where $N(R)$ is the smallest
normal subgroup of the free group $F(\Sigma)$ that contains all elements $uv^{-1}$ with
$(u=v) \in R$. In particular
$F(\Sigma) = \langle \Sigma \mid \emptyset \rangle$.
Of course, one can assume that all relations are of the form
$r=1$. In fact, usually the set of relations is given by a set of {\em relators}
$R \subseteq (\Sigma^{\pm 1})^+$, which corresponds to the set $\{ r = 1 \mid r
\in R\}$ of relations.

The \textit{free product} of two groups $G_1$ and $G_2$ is denoted by
$G_1*G_2$. If $G_i \simeq \langle \Sigma_i \mid R_i \rangle$ for $i \in \{1,2\}$
with $\Sigma_1 \cap \Sigma_2 = \emptyset$, then
$G_1 * G_2 \simeq \langle \Sigma_1 \cup \Sigma_2 \mid R_1 \cup R_2 \rangle$. 

The following transformations on group presentations (in either direction)
are known as \textit{Tietze transformations}: 
\begin{alignat*}{2}
(\Sigma, R) & \leftrightarrow  (\Sigma, R \cup \{u=v\}) \quad & & \text{if } uv^{-1} \in
N(R) \\
(\Sigma, R) & \leftrightarrow  (\Sigma \cup \{a\}, R \cup \{a=w\}) 
\quad & & \text{if } a \not\in \Sigma^{\pm 1}, w \in (\Sigma^{\pm 1})^* 
\end{alignat*}
If $(\Sigma', R')$ can be obtained by a sequence of Tietze transformations
from $(\Sigma, R)$, then $\langle \Sigma \mid R \rangle \simeq 
\langle \Sigma' \mid R' \rangle$ \cite{LySch77}.

\subsection{Straight-line programs} \label{S SLP}
 
We are using straight-line programs as a compressed representation of strings
with reoccuring subpatterns \cite{PlRy99}. A \textit{straight-line program (SLP) over the alphabet
$\Gamma$} is a context free grammar $\dA=(V,\Gamma,S,P)$, where $V$ is the set
of \textit{nonterminals}, $\Gamma$ is the set of \textit{terminals}, $S\in V$ is
the \textit{initial nonterminal}, and $P\subseteq V \times (V\cup \Gamma)^*$ is
the set of \textit{productions} such that (i) for every $X\in V$ there is
exactly one $\alpha \in (V \cup \Gamma)^*$ with $(X,\alpha)\in P$ and (ii) there
is no cycle in the relation $\{(X,Y)\in V\times V \mid \exists \alpha : (X,\alpha)
\in P, Y\in \text{alph}(\alpha)\}$. 
A production $(X,\alpha)$ is also written as
$X\rightarrow \alpha$. The language generated by the SLP $\dA$ contains exactly
one word $\eval(\dA)$. Moreover, every nonterminal $X\in V$ generates exactly
one word that is denoted by $\eval(\dA,X)$, or briefly $\eval(X)$, if $\dA$
is clear from the context. The size of $\dA$ is $|\dA|=
\sum_{(X,\alpha)\in P}|\alpha|$. It can be seen easily that an SLP can be
transformed in polynomial time
into an SLP in \textit{Chomsky normal form}, which means that
all productions
have the form $A\rightarrow BC$ or $A\rightarrow a$ for $A,B,C\in V$ and $a \in
\Gamma$. The following tasks can be solved in polynomial time. Except 
for the last one, proofs are straightforward.
\begin{itemize}
\item Given an SLP $\dA$, calculate $|\eval(\dA)|$.
\item Given an SLP $\dA$ and a natural number $i\leq |\eval(\dA)|$, calculate
$\eval(\dA)[i]$.
\item Given SLPs $\dA$ and $\dB$ decide whether $\eval(\dA)=\eval(\dB)$ \cite{Pla94}.
\end{itemize}
A \textit{deterministic rational transducer} is a 5-tuple
$T=(\Sigma,\Gamma,Q,\delta,q_0,F)$, where $\Sigma$ is the input alphabet, $\Gamma$
is the output alphabet, $Q$ is the set of states, 
$\delta:Q\times \Sigma \rightarrow Q\times \Gamma^*$ is the partial 
transition function, $q_0\in Q$ is the initial
state, and  $F \subseteq Q$ is the set of final states.
Let $\widehat{\delta} : Q \times \Sigma^* \to Q \times \Gamma^*$ be the
canonical extension of $\delta$. The partial mapping defined by
$T$ is $[\![ T ]\!] = \{ (u,v) \in \Sigma^* \times \Gamma^* \mid 
\widehat{\delta}(q_0,u) \in F \times\{v\} \}$. 
A proof of the following lemma can be found in \cite{BeChRa08}.

\begin{lemma} \label{L transducer}
{}From a given SLP $\dA$ and a given deterministic rational transducer $T$ we can
compute in polynomial time an SLP for the string $[\![ T ]\!](\eval(\dA))$ (if it is defined).
\end{lemma}
Let $G$ be a finitely generated group and $\Sigma$ a finite generating set
for $G$. The \textit{compressed word problem} for $G$ with respect to $\Sigma$ is
the following decision problem:

\medskip

\noindent
INPUT: An SLP $\dA$ over the terminal alphabet $\Sigma^{\pm1}$.

\smallskip

\noindent
OUTPUT: Does $\eval(\dA) =1$ hold in $G$?

\medskip

\noindent
In this problem, the input size is $|\dA|$.
As for the ordinary word problem, the complexity of the compressed word
problem does not depend on the chosen generating set.  This
allows one to speak of the compressed word problem for
the group $G$. The compressed word problem for $G$ is 
also denoted by $\CWP(G)$.

A \textit{composition system} $\dA=(V,\Gamma,S,P)$ is an SLP, which
additionally
allows productions of the form $A\rightarrow B[i:j]$ where $1\leq i\leq j\leq
|\eval(B)|$ \cite{GaKaPlRy96}. For such
a production we define $\eval(A)=\eval(B)[i:j]$. In \cite{Hag00}, Hagenah
presented a polynomial time algorithm that transforms a given composition system
into an SLP that generates the same word.

\subsection{Polynomial time Turing-reductions} \label{S reduction}

For two computational problems $A$ and $B$, we write
$A \leq_T^P B$ if $A$ is polynomial time Turing-reducible 
to $B$. This means that $A$ can be decided by a deterministic
polynomial time Turing-machine that uses $B$ as an oracle.
Clearly, $\leq_T^P$ is transitive, and
$A \leq_T^P B \in \mathsf{P}$ implies $A \in \mathsf{P}$.
More generally, if $A, B_1,\ldots,B_n$ are computational
problems, then we write $A \leq_T^P \{ B_1,\ldots, B_n\}$
if $A \leq_T^P \bigcup_{i=1}^n (\{i\} \times B_i)$.

\subsection{HNN-extensions}
\label{hnn_extensions}

Let us fix throughout this section a {\em base group} $H = \langle \Sigma \mid R \rangle$.
Let us also fix isomorphic subgroups $A_i, B_i \leq H$
($1 \leq i \leq n$) and isomorphisms $\varphi_i:
A_i \rightarrow B_i$. 
Let $h : (\Sigma^{\pm 1})^* \to H$ be the canonical
morphism, which maps a word $w \in (\Sigma^{\pm 1})^*$
to the element of $H$ it represents.
We consider the {\em HNN-extension}
\begin{equation}
G=\langle H,t_1, \ldots, t_n \mid  a^{t_i} = \varphi_i(a) \ (1 \leq i \leq
n, a \in A_i) \rangle.
\label{THE_hnn_extension}
\end{equation}
This means that  
$G=\langle \Sigma \cup \{t_1, \ldots, t_n\} \mid  R \cup
\{a^{t_i} = \varphi_i(a) \mid 1 \leq i \leq n, a \in A_i\} \rangle$.
It is known that the base group $H$ naturally embeds into
$G$ \cite{HiNeNe49}.
In this paper, we will be only concerned with 
the case that all groups $A_1,\ldots, A_n$ are finite and that
$\Sigma$ is finite. In this situation, we may assume that 
$\bigcup_{i=1}^n (A_i \cup B_i)\subseteq \Sigma$.
We say that $A_i$ and $B_i$ are {\em associated subgroups} in 
the HNN-extension $G$.
For the following, the notations $A_i(+1) = A_i$ and $A_i(-1) = B_i$ are useful.
Note that $\varphi_i^\alpha : A_i(\alpha) \to A_i(-\alpha)$ for $\alpha \in
\{+1,-1\}$.

We say that a word $u \in ( \Sigma^{\pm 1} \cup \{t_1,t_1^{-1}, \ldots, t_n, t_n^{-1}\})^*$ is {\em reduced}
if $u$ does not contain a factor of the form
$t_i^{-\alpha} w t_i^\alpha$ for $\alpha\in\{1,-1\}$,
$w \in (\Sigma^{\pm 1})^*$ and $h(w) \in A_i(\alpha)$.
With $\Red(H, \varphi_1, \ldots, \varphi_n)$ we denote
the set of all reduced words. 
For a word $u \in ( \Sigma^{\pm 1} \cup \{t_1,t_1^{-1}, \ldots, t_n, t_n^{-1}\})^*$
let us denote with $\pi_t(u)$ the projection of $u$ to the alphabet
$\{t_1,t_1^{-1}, \ldots, t_n, t_n^{-1}\}$.
The following Lemma provides a necessary and sufficient 
condition for equality of reduced strings in an HNN-extension \cite{LohSen08}:

\begin{lemma}  \label{equality reduced}
Let $u = u_0 t_{i_1}^{\alpha_1} u_1  \cdots  t_{i_\ell}^{\alpha_\ell} u_\ell$ and
$v = v_0 t_{j_1}^{\beta_1} v_1  \cdots t_{j_m}^{\beta_m} v_m$
be reduced words with 
$u_0, \ldots, u_\ell, v_0,\ldots, v_m \in  ( \Sigma^{\pm 1})^*$,
$\alpha_1, \ldots, \alpha_\ell, \beta_1,\ldots,\beta_m \in \{1,-1\}$, and
$i_1, \ldots, i_\ell, j_1,\ldots,j_m \in \{1,\ldots,n\}$.
Then $u=v$ in the HNN-extension $G$
from (\ref{THE_hnn_extension}) if and only if the following hold:
\begin{itemize}
\item $\pi_t(u) = \pi_t(v)$ (i.e., $\ell= m$, $i_k=j_k$, and
$\alpha_k=\beta_k$ for $1 \leq k \leq \ell$)
\item  there exist $c_1, \ldots, c_{2m} \in \bigcup_{k=1}^n (A_k \cup B_k)$ such that:
\begin{itemize}
\item $u_k c_{2k+1} = c_{2k} v_k$ in $H$ for $0 \leq k \leq \ell$ (here we set $c_0 = c_{2\ell+1}=1$)
\item $c_{2k-1} \in A_{i_k}(\alpha_k)$ and $c_{2k} =
\varphi^{\alpha_k}_{i_k}(c_{2k-1}) \in A_{i_k}(-\alpha_k)$ for $1 \leq k \leq
\ell$.
\end{itemize}
\end{itemize}
\end{lemma}
The second condition of the lemma can be visualized by  a diagram of the following
form (also called a Van Kampen diagram, see \cite{LySch77} for more details), where $\ell=m=4$.
Light-shaded (resp. dark-shaded) faces represent relations in
$H$ (resp. relations of the form $c t_i^\alpha = t_i^\alpha
\varphi_i^\alpha(c)$ with $c \in A_i(\alpha)$).
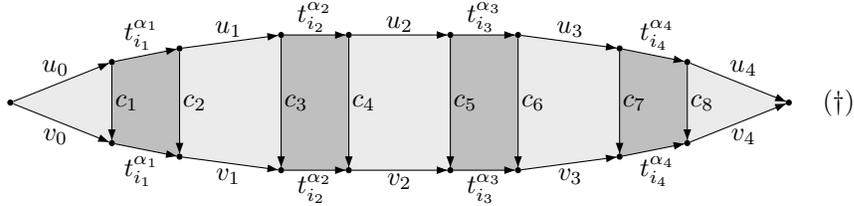
\begin{center}
\setlength{\unitlength}{.9mm}
\begin{picture}(115,26)(0,-3)
\gasset{Nframe=n,AHnb=1,ELdist=0.5}
\gasset{Nw=.9,Nh=.9,Nfill=y}
\put(120,9){$(\dagger)$}
\drawpolygon[fillgray=0.92](0,10)(15,16)(15,4)
\drawpolygon[fillgray=0.75](15,4)(15,16)(25,18)(25,2)
\drawpolygon[fillgray=0.92](25,2)(25,18)(40,20)(40,0)
\drawpolygon[fillgray=0.75](40,0)(40,20)(50,20)(50,0)
\drawpolygon[fillgray=0.92](50,0)(50,20)(65,20)(65,0)
\drawpolygon[fillgray=0.75](65,0)(65,20)(75,20)(75,0)
\drawpolygon[fillgray=0.92](90,2)(90,18)(75,20)(75,0)
\drawpolygon[fillgray=0.75](100,4)(100,16)(90,18)(90,2)
\drawpolygon[fillgray=0.92](115,10)(100,16)(100,4)
\node(0)(0,10){}
\node(a)(15,16){}
\node(b)(25,18){}
\node(c)(40,20){}
\node(d)(50,20){}
\node(e)(65,20){}
\node(f)(75,20){}
\node(g)(90,18){}
\node(h)(100,16){}
\node(i)(115,10){}
\node(a')(15,4){}
\node(b')(25,2){}
\node(c')(40,0){}
\node(d')(50,0){}
\node(e')(65,0){}
\node(f')(75,0){}
\node(g')(90,2){}
\node(h')(100,4){}
\drawedge(0,a){$u_0$}
\drawedge(a,b){$t_{i_1}^{\alpha_1}$}
\drawedge(b,c){$u_1$}
\drawedge(c,d){$t_{i_2}^{\alpha_2}$}
\drawedge(d,e){$u_2$}
\drawedge(e,f){$t_{i_3}^{\alpha_3}$}
\drawedge(f,g){$u_3$}
\drawedge(g,h){$t_{i_4}^{\alpha_4}$}
\drawedge(h,i){$u_4$}
\drawedge[ELside=r](0,a'){$v_0$}
\drawedge[ELside=r](a',b'){$t_{i_1}^{\alpha_1}$}
\drawedge[ELside=r](b',c'){$v_1$}
\drawedge[ELside=r](c',d'){$t_{i_2}^{\alpha_2}$}
\drawedge[ELside=r](d',e'){$v_2$}
\drawedge[ELside=r](e',f'){$t_{i_3}^{\alpha_3}$}
\drawedge[ELside=r](f',g'){$v_3$}
\drawedge[ELside=r](g',h'){$t_{i_4}^{\alpha_4}$}
\drawedge[ELside=r](h',i){$v_4$}
\drawedge(a,a'){$c_1$}
\drawedge[ELpos=50](b,b'){$c_2$}
\drawedge[ELpos=50](c,c'){$c_3$}
\drawedge[ELpos=50](d,d'){$c_4$}
\drawedge[ELpos=50](e,e'){$c_5$}
\drawedge[ELpos=50](f,f'){$c_6$}
\drawedge[ELpos=50](g,g'){$c_7$}
\drawedge[ELpos=50](h,h'){$c_8$}
\end{picture}
\end{center}
The elements $c_1, \ldots, c_{2\ell}$ in such a diagram 
are also called \emph{connecting elements}.

When solving the compressed word problem for HNN-extensions we will make 
use of the following simple lemma, which allows us to transform an arbitrary
string over the generating set of an HNN-extension into a reduced one.

\begin{lemma} \label{lemma reducing}
Assume that $u = u_0 t_{i_1}^{\alpha_1} u_1  \cdots  t_{i_n}^{\alpha_n} u_n$
and $v = v_0 t_{j_1}^{\beta_1} v_1  \cdots t_{j_m}^{\beta_m} v_m$ are reduced
strings. 
Let $d(u,v)$ be the largest number
$d \geq 0$ such that
\begin{enumerate}[(a)]
\item $A_{i_{n-d+1}}(\alpha_{n-d+1}) = A_{j_d}(-\beta_d)$ (we set
$A_{i_{n+1}}(\alpha_{n+1}) = A_{j_0}(-\beta_0) = 1$) and
\item $\exists c \in A_{j_d}(-\beta_d) : t_{i_{n-d+1}}^{\alpha_{n-d+1}} u_{n-d+1}  \cdots  
t_{i_n}^{\alpha_n} \, u_n\, v_0\, t_{j_1}^{\beta_1} \cdots
v_{d-1}\,t_{j_d}^{\beta_d} = c$ in the group $G$ from (\ref{THE_hnn_extension})
(note that this  condition is satisfied for $d=0$).
\end{enumerate}
Moreover, let $c(u,v) \in A_{j_d}(-\beta_d)$ be the element 
$c$ in (b) (for $d=d(u,v)$). Then 
$$
u_0 t_{i_1}^{\alpha_1} u_1  \cdots  t_{i_{n-d(u,v)}}^{\alpha_{n-d(u,v)}} (u_{n-d(u,v)}\,
c(u,v)\,  v_{d(u,v)}) t_{j_{d(u,v)+1}}^{\beta_{d(u,v)+1}} v_{d(u,v)+1}  \cdots t_{j_m}^{\beta_m} v_m
$$
is a reduced string equal to $uv$ in $G$.
\end{lemma}
\noindent
The above lemma can be visualized by the following diagram.

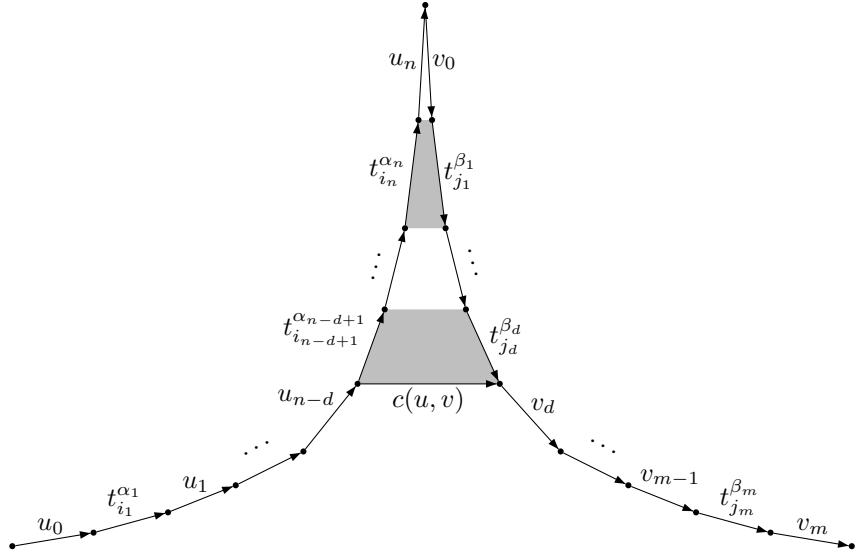
\begin{center}
\setlength{\unitlength}{.9mm}
\begin{picture}(115,100)(0,-3)
\gasset{Nframe=n,AHnb=1,ELdist=0.5}
\gasset{Nw=.9,Nh=.9,Nfill=y}
\drawpolygon[fillgray=0.75](51,34)(55,45)(67,45)(72,34) 
\drawpolygon[fillgray=0.75](58,57)(60,73)(62,73)(64,57)
\node(a)(0,10){}
\node(b)(12,12){}
\node(c)(23,15){}
\node(d)(33,19){}
\node(e)(43,24){}
\node(f)(51,34){}
\node(g)(55,45){}
\node(h)(58,57){}
\node(i)(60,73){}
\node(j)(61,90){}
\node(k)(62,73){}
\node(l)(64,57){}
\node(m)(67,45){}
\node(n)(72,34){}
\node(o)(81,24){}
\node(p)(91,19){}
\node(q)(101,15){}
\node(r)(112,12){}
\node(s)(124,10){}
\drawedge(a,b){$u_0$}
\drawedge(b,c){$t^{\alpha_1}_{i_1}$}
\drawedge(c,d){$u_1$}
\drawedge(d,e){\begin{turn}{27}$\cdots$\end{turn}}
\drawedge(e,f){$u_{n-d}$}
\drawedge(f,g){$t^{\alpha_{n-d+1}}_{i_{n-d+1}}$}
\drawedge(g,h){\begin{turn}{75}$\cdots$\end{turn}}
\drawedge(h,i){$t^{\alpha_{n}}_{i_n}$}
\drawedge(i,j){$u_n$}
\drawedge(j,k){$v_0$}
\drawedge(k,l){$t_{j_1}^{\beta_1}$}
\drawedge(l,m){\begin{turn}{-75}$\cdots$\end{turn}}
\drawedge(m,n){$t_{j_d}^{\beta_d}$}
\drawedge(n,o){$v_d$}
\drawedge(o,p){\begin{turn}{-27}$\cdots$\end{turn}}
\drawedge(p,q){$v_{m-1}$}
\drawedge(q,r){$t^{\beta_m}_{j_m}$}
\drawedge(r,s){$v_m$}
\drawedge[ELside=r](f,n){$c(u,v)$}
\end{picture}
\end{center}

\subsection{Some simple compressed word problems}

We will use the following theorem on free products $G_1 * G_2$
that was shown in \cite{LoSchl07}.

\begin{theorem} \label{T free}
$\CWP(G_1 * G_2) \leq_T^P \{\CWP(G_1), \CWP(G_2)\}$.
\end{theorem}
For our reduction of the compressed word problem of an HNN-extension
to the compressed word problem of the base group, we need the special
case that in (\ref{THE_hnn_extension}) we have $H = A_1 = \cdots = A_n = B_1 =
\cdots = B_n$ (in particular, $H$ is finite). 
In this case, we can even assume that the finite group $H$
(represented by its multiplication table) is part of the input:

\begin{lemma}\label{cwp semidirect}
The following problem can be solved in polynomial time:

\medskip

\noindent
INPUT: A finite group $H$,
automorphisms $\varphi_i : H \to H$ ($1 \leq i \leq n$),
and an SLP $\dA$ over the alphabet
$H \cup \{t_1, t_1^{-1}, \ldots t_n, t_n^{-1} \}$.

\smallskip

\noindent
QUESTION: $\eval(\dA)=1$ in  $\langle H, t_1, \ldots, t_n \mid h^{t_i} =
\varphi_i(h) \ (1 \leq i \leq n, h \in H) \rangle$?
\end{lemma}

\begin{proof}
Let $s \in (H \cup \{t_1, t_1^{-1}, \ldots t_n,
t_n^{-1} \})^*$. From the defining equations of the group
$G = \langle H, t_1, \ldots, t_n \mid h^{t_i} =
\varphi_i(h) \ (1 \leq i \leq n, h \in H) \rangle$ it follows
that there exists a unique $h \in H$ with $s = \pi_t(s) h$ in $G$.
Hence, $s=1$ in $G$ if and only if $\pi_t(s) = 1$ in the free
group $F(t_1,\ldots,t_n)$ and $h=1$ in $H$.

Now, let $\dA$ be an SLP over the alphabet
$H \cup \{t_1, t_1^{-1}, \ldots t_n, t_n^{-1} \}$. W.l.o.g. assume that
$\dA$ is in Chomsky normal form. 
It is straightforward to compute an SLP for the projection $\pi_t(\eval(\dA))$.
Since by Theorem~\ref{T free} the word problem for the free group
$F(t_1,\ldots,t_n)$ can be solved in polynomial time, it suffices
to compute for every nonterminal $A$ of $\dA$ the unique $h_A \in H$
with $\eval(A) = \pi_t(\eval(A)) h_A$ in $G$.
We compute the elements $h_A$ bottom up. The case that the right-hand
side for $A$ is a terminal symbol from 
$H \cup \{t_1, t_1^{-1}, \ldots t_n, t_n^{-1} \}$
is clear. Hence, assume that $A \to BC$ is a production of $\dA$ and 
assume that $h_B, h_C \in H$ are already computed.
Hence, in $G$ we have:
$$
\eval(A) = \eval(B) \eval(C) = \pi_t(\eval(B)) h_B \pi_t(\eval(C)) h_C.
$$
Thus, it suffices to compute the unique $h \in H$ with 
$h_B \pi_t(\eval(C)) = \pi_t(\eval(C))  h$ in $G$. 
Note that if 
$\pi_t(\eval(C)) = t_{i_1}^{\alpha_1}t_{i_2}^{\alpha_2} \cdots
t_{i_n}^{\alpha_n}$, then
$$
h = \varphi_{i_n}^{\alpha_n}( \cdots 
\varphi_{i_2}^{\alpha_2}(\varphi_{i_1}^{\alpha_1}(h_B)) \cdots ) = 
(\varphi_{i_1}^{\alpha_1} \circ \cdots \circ
\varphi_{i_n}^{\alpha_n})(h_B).
$$
The automorphism $f = \varphi_{i_1}^{\alpha_1} \circ \cdots
\circ \varphi_{i_n}^{\alpha_n}$ can be easily computed
from an SLP $\dC$ for the string $\pi_t(\eval(C))$
by replacing in $\dC$ the terminal symbol $t_i$ (resp. $t_i^{-1}$) by
$\varphi_i$ (resp. $\varphi_i^{-1}$). This allows to compute 
$f$ bottom-up and then to compute $f(h_B)$.
\qed
\end {proof}
Note that the group 
$\langle H, t_1, \ldots, t_n \mid h^{t_i} =
\varphi_i(h) \ (1 \leq i \leq n, h \in H) \rangle$ is the semidirect
product $H \rtimes_{\varphi} F$,
where $F = F(t_1, \ldots, t_n)$ is the free
group generated by $t_1, \ldots, t_n$ and the 
homomorphism 
$\varphi : F \to \text{Aut}(H)$ is defined 
by $\varphi(t_i) = \varphi_i$.

\section{Compressed word problem of an HNN-extension} \label{S CWP HNN}

In this section we will prove that the compressed word 
problem for an HNN-extension of the form (\ref{HNN}) 
is polynomial time Turing-reducible to the compressed word problem for $H$.
In fact, we will prove the existence of such a reduction for 
a slightly more general problem, which we introduce below.

For the further consideration, let us fix the finitely generated
group $H$ together with the finite subgroups $A$ and $B$.
Let $\Sigma$ be a finite generating set for $H$.
These data are fixed, i.e., they will not belong to the input
of computational problems. 

In the following, when writing down a multiple HNN-extension
\begin{equation}\label{mult HNN}
\langle H, t_1, \ldots, t_n \mid a^{t_i} = \varphi_i(a) \ (1 \leq i \leq n, a \in A) \rangle,
\end{equation}
we assume implicitly that 
every $\varphi_i$ is in fact an isomorphism between  
subgroups $A_1 \leq A$ and $B_1 \leq B$. Hence,
$\varphi_i$ can be viewed as a {\em partial} isomorphism from our 
fixed subgroup $A$ to our fixed subgroup $B$, and
(\ref{mult HNN}) is in fact an abbreviation for the group
$$\langle H, t_1, \ldots, t_n \mid a^{t_i} = \varphi_i(a) \ (1 \leq i \leq n, 
a \in \dom(\varphi_i)) \rangle.
$$
Note that there is only a fixed
number of partial isomorphisms from $A$ to $B$, but we 
allow $\varphi_i = \varphi_j$ for $i \neq j$ in (\ref{mult HNN}).

Let us introduce several restrictions and extensions of $\CWP(G)$.
Our most general problem is the following computational problem
$\UCWP(H,A,B)$ (the letter ``U'' stands for ``uniform'', meaning that
a list of partial isomorphisms from $A$ to $B$ is part of the input):

\medskip

\noindent
INPUT:  Partial isomorphisms $\varphi_i : A \to B$ ($1 \leq i \leq n$)
and an SLP $\dA$ over the alphabet
$\Sigma^{\pm 1} \cup \{t_1, t_1^{-1}, \ldots, t_n, t_n^{-1} \}$.

\smallskip

\noindent
QUESTION: $\eval(\dA)=1$ in $\langle H, t_1, \ldots, t_n \mid a^{t_i} =
\varphi_i(a) \ (1 \leq i \leq n, a \in A) \rangle$?

\medskip

\noindent
The restriction of this problem $\UCWP(H,A,B)$ 
to reduced input strings is 
denoted by $\RUCWP(H,A,B)$. It is formally defined as 
the following problem:

\medskip

\noindent
INPUT:  Partial isomorphisms $\varphi_i : A \to B$ ($1 \leq i \leq n$)
and SLPs $\dA,\dB$ over the alphabet
$\Sigma^{\pm 1} \cup \{t_1, t_1^{-1}, \ldots, t_n, t_n^{-1} \}$
such that $\eval(\dA), \eval(\dB) \in \Red(H, \varphi_1, \ldots, \varphi_n)$.

\smallskip

\noindent
QUESTION: $\eval(\dA)=\eval(\dB)$ in $\langle H, t_1, \ldots, t_n \mid 
a^{t_i} = \varphi_i(a) \ (1 \leq i \leq n, a \in A) \rangle$?

\medskip

\noindent
Let us now consider a fixed list of partial isomorphisms
$\varphi_1,\ldots,\varphi_n : A \to B$. Then
$\RCWP(H,A,B, \varphi_1,\ldots, \varphi_n)$  is the following computational problem:

\medskip

\noindent
INPUT: Two SLPs $\dA$ and $\dB$ over the alphabet
$\Sigma^{\pm 1} \cup \{t_1, t_1^{-1}, \ldots, t_n, t_n^{-1} \}$
such that $\eval(\dA), \eval(\dB) \in \Red(H, \varphi_1, \ldots, \varphi_n)$.

\smallskip

\noindent
QUESTION: $\eval(\dA)=\eval(\dB)$ in $\langle H, t_1, \ldots, t_n \mid a^{t_i} = \varphi_i(a)\ (1 \leq i
\leq n, a \in A) \rangle$?

\medskip

\noindent
Our main result is:

\begin{theorem}\label{T real main}
$\UCWP(H,A,B) \leq_P^T \CWP(H)$.
\end{theorem}
The rest of Section~\ref{S CWP HNN} 
is concerned with the proof of Theorem~\ref{T real main}.

\subsection{Reducing to reduced sequences} \label{S reduced}

First we show that we may restrict ourselves to SLPs that evaluate to reduced
strings:

\begin{lemma} \label{L reduced}
$\UCWP(H,A,B) \leq_P^T \RUCWP(H,A,B)$. More precisely, there is a 
polynomial time Turing-reduction from 
$\UCWP(H,A,B)$ to $\RUCWP(H,A,B)$ that on input 
$(\varphi_1,\ldots, \varphi_n, \dA)$ only asks
$\RUCWP(H,A,B)$-queries of the form
$(\varphi_1,\ldots, \varphi_n, \dA', \dB')$
(thus, the list of partial isomorphisms is not changed).
\end{lemma}

\begin{proof}
Consider partial isomorphisms $\varphi_i : A \to B$ ($1 \leq i \leq n$)
and let 
$$G = \langle H, t_1, \ldots, t_n \mid 
a^{t_i} = \varphi_i(a) \ (1 \leq i \leq n, a \in A) \rangle.$$
Moreover, let $\dA$ be an 
SLP in Chomsky normal form over the alphabet
$\Sigma^{\pm 1} \cup \{t_1, t_1^{-1}, \ldots, t_n, t_n^{-1} \}$.
Using oracle access to $\RUCWP(H,A,B)$,
we will construct bottom-up a {\em composition system}
$\dA'$ with $\eval(\dA')=\eval(\dA)$ in $G$ and
$\eval(\dA')$ reduced, on which finally 
the $\RUCWP(H,A,B)$-oracle can be
asked whether $\eval(\dA')=1$ in $G$.
The system $\dA'$ has the same variables as $\dA$ but
for every variable $X$, $\eval(\dA',X)$ is reduced
and $\eval(\dA',X)=\eval(\dA,X)$ in $G$. 

Assume that $X \to YZ$ is a production of $\dA$, where
$Y$ and $Z$ were already processed during our bottom-up
reduction process. Hence, $\eval(Y)$ and $\eval(Z)$ are reduced.
Let 
$$
\eval(Y)  =  u_0 t_{i_1}^{\alpha_1} u_1  \cdots  t_{i_\ell}^{\alpha_\ell} u_\ell 
\quad \text{and} \quad 
\eval(Z)  =  v_0 t_{j_1}^{\beta_1} v_1  \cdots t_{j_m}^{\beta_m} v_m.
$$
with $u_0,\ldots,u_\ell, v_0,\ldots,v_m \in (\Sigma^{\pm 1})^*$.
For $1 \leq k \leq \ell$ (resp. $1 \leq k \leq m$) let
$p(k)$ (resp. $q(k)$) be the unique position within $\eval(Y)$
(resp. $\eval(Z)$) such that 
$\eval(Y)[:p(k)] =  u_0 t_{i_1}^{\alpha_1} u_1  \cdots  t_{i_k}^{\alpha_k}$
(resp. $\eval(Z)[:q(k)] = v_0 t_{j_1}^{\beta_1} v_1  \cdots
t_{j_k}^{\beta_k}$). These positions can be computed in polynomial time
from $k$ using simple arithmetic.

According to Lemma~\ref{lemma reducing} it suffices to find
$d=d(\eval(Y), \eval(Z)) \in \mathbb{N}$ 
and $c=c(\eval(Y),\eval(Z)) \in  A \cup B$
in polynomial time.
This can be done, using binary search:
First, compute $\min\{l,m\}$.
For a given number $k \leq \min\{\ell,m\}$ we want to check
whether 
\begin{eqnarray}
\label{gl p}
t_{i_{\ell-k+1}}^{\alpha_{\ell-k+1}} u_{\ell-k+1}  \cdots  
t_{i_\ell}^{\alpha_\ell} \, u_\ell\, v_0\, t_{j_1}^{\beta_1} \cdots
v_{k-1} t_{j_k}^{\beta_k} \in A_{i_{\ell-k+1}}(\alpha_{\ell-k+1}) = A_{j_k}(-\beta_k)
\end{eqnarray}
in the group $G$. Note that (\ref{gl p}) is equivalent to
$t_{i_{\ell-k+1}}^{\alpha_{\ell-k+1}} = t_{j_k}^{-\beta_k}$ and 
\begin{eqnarray}
\label{gl c}
\bigvee_{c \in A_{j_k}(-\beta_k)} 
\eval(Y)[p(\ell-k+1):]^{-1} c = \eval(Z)[:q(k)] .
\end{eqnarray}
The two sides of this equation are reduced strings and the 
number of possible values $c \in A_{j_k}(-\beta_k)$ is bounded by
a constant. Hence, (\ref{gl c}) is equivalent to a 
constant number of $\RUCWP(H,A,B)$-instances that can
be computed in polynomial time.

In order to find with binary search
the value $d$ (i.e. the largest $k \geq 0$ 
such that (\ref{gl p}) holds), one has to observe 
that (\ref{gl p}) implies that (\ref{gl p}) also holds
for every smaller value $k$ (this follows from Lemma~\ref{equality reduced}).
{}From $d$, we can compute in polynomial
time positions $p(\ell-d+1)$ and $q(d)$. Then, according to 
Lemma~\ref{lemma reducing}, the string
$$\eval(Y)[:p(\ell-d+1)-1]\, c\, \eval(Z)[q(d)+1:]$$ 
is reduced and equal
to $\eval(Y)\eval(Z)$ in $G$.
Hence, we can replace the production $X \to YZ$ by
$X \to Y[:p(\ell-d+1)-1]\, c\, Z[q(d)+1:]$.
\qed
\end{proof}
The above proof can be also used in order to derive:

\begin{lemma} \label{L reduced 2}
Let $\varphi_1,\ldots,\varphi_n : A \to B$ 
be fixed partial isomorphisms. Then
$\CWP(\langle H, t_1, \ldots, t_n \mid a^{t_i} = \varphi_i(a)\ (1 \leq i
\leq n, a \in A) \rangle)$ is polynomial time Turing-reducible to 
$\RCWP(H,A,B, \varphi_1,\ldots, \varphi_n)$.
\end{lemma}

\subsection{Reduction to a constant number of stable letters} 
\label{S constant number}

In this section, we show that the number of different
stable letters can be reduced to a constant. For this, it is 
important to note that the associated
subgroups $A, B \leq H$ do not belong to the input; so their size is 
a fixed constant.

Fix the constant $\delta = 2 \cdot |A|! \cdot 2^{|A|}$ for the rest
of the paper. Note that the number of HNN-extensions of the form 
$\langle H, t_1, \ldots, t_k \mid a^{t_i} = \psi_i(a)\ (1 \leq i \leq k, a \in A) \rangle$
with $k \leq \delta$ is constant. 
The following lemma says that $\RUCWP(H,A,B)$ can be reduced 
in polynomial time to one of the problems $\RCWP(H,A,B,\psi_1,\ldots,\psi_k)$.
Moreover, we can determine in polynomial time, which of these problems
arises.

\begin{lemma} \label{L reducing to constant}
There exists a polynomial time algorithm for the following:

\medskip

\noindent
INPUT:  Partial isomorphisms $\varphi_1, \ldots, \varphi_n : 
A \to B$ and SLPs $\dA,\dB$ over the alphabet
$\Sigma^{\pm 1} \cup \{t_1, t_1^{-1}, \ldots t_n, t_n^{-1} \}$
such that $\eval(\dA), \eval(\dB) \in \Red(H, \varphi_1, \ldots, \varphi_n)$.

\smallskip

\noindent
OUTPUT: Partial isomorphisms $\psi_1,\ldots,\psi_k : A \to B$ 
where $k \leq \delta$ and SLPs $\dA'$, $\dB'$ over
the alphabet $\Sigma^{\pm 1} \cup \{t_1, t_1^{-1}, \ldots t_k, t_k^{-1} \}$
such that:
\begin{itemize}
\item For every $1 \leq i \leq k$ there exists $1 \leq j \leq n$
with $\psi_i = \varphi_j$.
\item $\eval(\dA'), \eval(\dB') \in \Red(H,\psi_1,\ldots,\psi_k)$
\item $\eval(\dA)=\eval(\dB)$ in $\langle H, t_1, \ldots, t_n \mid a^{t_i} = 
\varphi_i(a)\ (1 \leq i \leq n, a \in A) \rangle$ if and only if
$\eval(\dA')=\eval(\dB')$ in $\langle H, t_1, \ldots, t_k 
\mid a^{t_i} = \psi_i(a)\ (1 \leq i \leq k, a \in A) \rangle$.
\end{itemize}
\end{lemma}

\begin{proof}
Fix an input $(\varphi_1, \ldots, \varphi_n, \dA, \dB)$ for the problem
$\RUCWP(H,A,B)$. In particular, $\eval(\dA), \eval(\dB) \in \Red(H,\varphi_1,\ldots,\varphi_n)$.
Define the function $\tau : \{1, \ldots, n\} \to  \{1, \ldots, n\}$ by 
$$
\tau(i) = \min \{ k \mid \varphi_k =
\varphi_i\}.
$$
This mapping can be easily computed in polynomial time from the sequence
$\varphi_1, \ldots, \varphi_n$. Assume w.l.o.g. that 
$\ran(\tau) = \{1, \ldots,\gamma\}$ for some $\gamma \leq n$.
Note that $\gamma \leq |A|! \cdot 2^{|A|} = \frac{\delta}{2}$.
For every $t_i$ ($1 \leq i \leq \gamma$)
we take two stable letters $t_{i,0}$ and $t_{i,1}$.
Hence, the total number of stable letters is at most $\delta$.
Moreover,
we define a sequential transducer $T$ which, reading as input the word
$u_0 t_{i_1}^{\alpha_1} u_1  \cdots  t_{i_m}^{\alpha_m} u_m$ 
(with $u_0,\ldots,u_m \in (\Sigma^{\pm 1})^+$ and 
$1 \leq i_1,\ldots,i_m \leq n$)  returns
$$
[\![ T ]\!](w) = u_0 \,t_{\tau(i_1),1}^{\alpha_1}\, u_1\, t_{\tau(i_2),0}^{\alpha_2}\, u_2\,
t_{\tau(i_3),1}^{\alpha_3}\, u_3 \cdots t_{\tau(i_m), m \text{ mod } 2}^{\alpha_m}\, u_m.
$$
Finally, we define the HNN-extension
$$
G' = \langle H, t_{1,0},t_{1,1},\ldots, t_{\gamma,0}, t_{\gamma,1} \mid  
a^{t_{i,k}} =\varphi_i(a) \
(1 \leq i \leq \gamma, k \in \{0,1\}, a \in A) \rangle.
$$
This HNN-extension has $2 \gamma \leq \delta$ many stable letters; it 
is the HNN-extension $\langle H, t_1, \ldots, t_k \mid a^{t_i} = \psi_i(a)\ (1 \leq i \leq
k, a \in A) \rangle$ from the lemma.

\medskip

\noindent
{\bf Claim:} Let 
$u,v  \in \Red(H, \varphi_1, \ldots, \varphi_n)$ be reduced. Then also
$[\![ T ]\!](u)$ and $[\![ T ]\!](v)$ are reduced.
Moreover, the following are equivalent:
\begin{enumerate}[(a)]
\item $u = v$ in  $\langle H, t_1, \ldots, t_n \mid a^{t_i} =
  \varphi_i(a)\ (1 \leq i \leq n, a \in A) \rangle$
\item $[\![ T ]\!](u) = [\![ T ]\!](v)$ in the HNN-extension $G'$ and 
$\pi_t(u)=\pi_t(v)$.    
\end{enumerate}

\noindent
{\it Proof of the claim.}
Let $u = u_0 t_{i_1}^{\alpha_1} u_1  \cdots  t_{i_\ell}^{\alpha_\ell} u_\ell$ and
$v = v_0 t_{j_1}^{\beta_1} v_1  \cdots t_{j_m}^{\beta_m} v_m$.
The first statement is obvious due to the fact that $[\![ T ]\!](u)$ does
not contain a subword of the form $t^\alpha_{i,k} w t^\beta_{j,k}$
for $k \in \{0,1\}$, and similarly for $[\![ T ]\!](v)$.

For $(a)\Rightarrow (b)$ note that by Lemma~\ref{equality reduced}, $u=v$ in 
$\langle H, t_1, \ldots, t_n \mid a^{t_i} = \varphi_i(a)\ (1 \leq i \leq n, a \in A) \rangle$ implies
$\pi_t(u)=\pi_t(v)$ (i.e.
$\ell=m$, $\alpha_1 = \beta_1, \ldots, \alpha_m=\beta_m$, 
$i_1 = j_1, \ldots, i_m=j_m$), and 
that there exists a Van Kampen diagram of the following form:
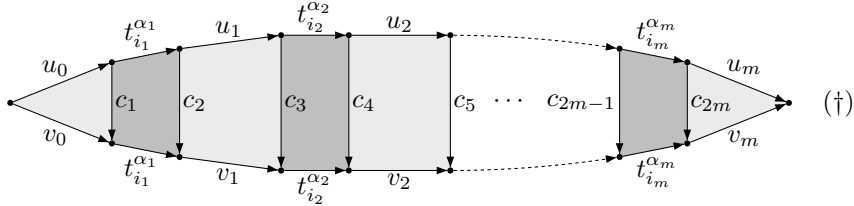
\begin{center}
\setlength{\unitlength}{.9mm}
\begin{picture}(115,26)(0,-3)
\gasset{Nframe=n,AHnb=1,ELdist=0.5}
\gasset{Nw=.9,Nh=.9,Nfill=y}
\put(120,9){$(\dagger)$}
\drawpolygon[fillgray=0.92](0,10)(15,16)(15,4)
\drawpolygon[fillgray=0.75](15,4)(15,16)(25,18)(25,2)
\drawpolygon[fillgray=0.92](25,2)(25,18)(40,20)(40,0)
\drawpolygon[fillgray=0.75](40,0)(40,20)(50,20)(50,0)
\drawpolygon[fillgray=0.92](50,0)(50,20)(65,20)(65,0)
\drawpolygon[fillgray=1](65,0)(65,20)(75,20)(75,0)
\drawpolygon[fillgray=1](90,2)(90,18)(75,20)(75,0)
\drawpolygon[fillgray=0.75](100,4)(100,16)(90,18)(90,2)
\drawpolygon[fillgray=0.92](115,10)(100,16)(100,4)
\node(0)(0,10){}
\node(a)(15,16){}
\node(b)(25,18){}
\node(c)(40,20){}
\node(d)(50,20){}
\node(e)(65,20){}
\node(g)(90,18){}
\node(h)(100,16){}
\node(i)(115,10){}
\node(a')(15,4){}
\node(b')(25,2){}
\node(c')(40,0){}
\node(d')(50,0){}
\node(e')(65,0){}
\node(g')(90,2){}
\node(h')(100,4){}
\drawedge(0,a){$u_0$}
\drawedge(a,b){$t_{i_1}^{\alpha_1}$}
\drawedge(b,c){$u_1$}
\drawedge(c,d){$t_{i_2}^{\alpha_2}$}
\drawedge(d,e){$u_2$}
\drawedge(g,h){$t_{i_m}^{\alpha_m}$}
\drawedge(h,i){$u_m$}
\drawedge[ELside=r](0,a'){$v_0$}
\drawedge[ELside=r](a',b'){$t_{i_1}^{\alpha_1}$}
\drawedge[ELside=r](b',c'){$v_1$}
\drawedge[ELside=r](c',d'){$t_{i_2}^{\alpha_2}$}
\drawedge[ELside=r](d',e'){$v_2$}
\drawedge[curvedepth=.5,dash={.5}1](e,g){}
\drawedge[curvedepth=-.5,dash={.5}1](e',g'){}
\drawedge[ELside=r](g',h'){$t_{i_m}^{\alpha_m}$}
\drawedge[ELside=r](h',i){$v_m$}
\drawedge(a,a'){$c_1$}
\drawedge[ELpos=50](b,b'){$c_2$}
\drawedge[ELpos=50](c,c'){$c_3$}
\drawedge[ELpos=50](d,d'){$c_4$}
\drawedge[ELpos=50](e,e'){$c_5\;\, \cdots$}
\drawedge[ELpos=50,ELside=r](g,g'){$c_{2m-1}$}
\drawedge[ELpos=50](h,h'){$c_{2m}$}
\end{picture}
\end{center}
The defining equations of $G'$ imply that 
the following is a valid Van Kampen diagram in $G'$:
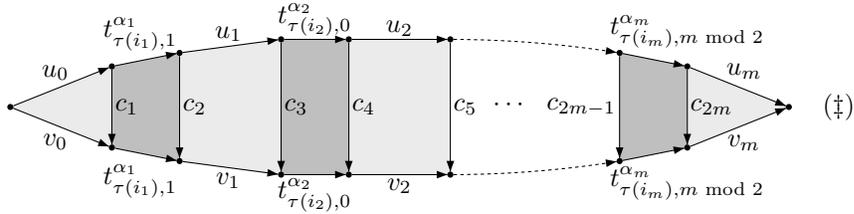
\begin{center}
\setlength{\unitlength}{.9mm}
\begin{picture}(115,30)(0,-5)
\gasset{Nframe=n,AHnb=1,ELdist=0.5}
\gasset{Nw=.9,Nh=.9,Nfill=y}
\put(120,9){$(\ddagger)$}
\drawpolygon[fillgray=0.92](0,10)(15,16)(15,4)
\drawpolygon[fillgray=0.75](15,4)(15,16)(25,18)(25,2)
\drawpolygon[fillgray=0.92](25,2)(25,18)(40,20)(40,0)
\drawpolygon[fillgray=0.75](40,0)(40,20)(50,20)(50,0)
\drawpolygon[fillgray=0.92](50,0)(50,20)(65,20)(65,0)
\drawpolygon[fillgray=1](65,0)(65,20)(75,20)(75,0)
\drawpolygon[fillgray=1](90,2)(90,18)(75,20)(75,0)
\drawpolygon[fillgray=0.75](100,4)(100,16)(90,18)(90,2)
\drawpolygon[fillgray=0.92](115,10)(100,16)(100,4)
\node(0)(0,10){}
\node(a)(15,16){}
\node(b)(25,18){}
\node(c)(40,20){}
\node(d)(50,20){}
\node(e)(65,20){}
\node(g)(90,18){}
\node(h)(100,16){}
\node(i)(115,10){}
\node(a')(15,4){}
\node(b')(25,2){}
\node(c')(40,0){}
\node(d')(50,0){}
\node(e')(65,0){}
\node(g')(90,2){}
\node(h')(100,4){}
\drawedge(0,a){$u_0$}
\drawedge(a,b){$t_{\tau(i_1),1}^{\alpha_1}$}
\drawedge(b,c){$u_1$}
\drawedge(c,d){$t_{\tau(i_2),0}^{\alpha_2}$}
\drawedge(d,e){$u_2$}
\drawedge[ELpos=90](g,h){$t_{\tau(i_m), m \text{ mod } 2}^{\alpha_m}$}
\drawedge(h,i){$u_m$}
\drawedge[ELside=r](0,a'){$v_0$}
\drawedge[ELside=r](a',b'){$t_{\tau(i_1),1}^{\alpha_1}$}
\drawedge[ELside=r](b',c'){$v_1$}
\drawedge[ELside=r](c',d'){$t_{\tau(i_2),0}^{\alpha_2}$}
\drawedge[ELside=r](d',e'){$v_2$}
\drawedge[curvedepth=.5,dash={.5}1](e,g){}
\drawedge[curvedepth=-.5,dash={.5}1](e',g'){}
\drawedge[ELside=r,ELpos=90](g',h'){$t_{\tau(i_m),m \text{ mod } 2}^{\alpha_m}$}
\drawedge[ELside=r](h',i){$v_m$}
\drawedge(a,a'){$c_1$}
\drawedge[ELpos=50](b,b'){$c_2$}
\drawedge[ELpos=50](c,c'){$c_3$}
\drawedge[ELpos=50](d,d'){$c_4$}
\drawedge[ELpos=50](e,e'){$c_5\;\, \cdots$}
\drawedge[ELpos=50,ELside=r](g,g'){$c_{2m-1}$}
\drawedge[ELpos=50](h,h'){$c_{2m}$}
\end{picture}
\end{center}
Hence, $[\![ T ]\!](u) = [\![ T ]\!](v)$ in $G'$.

For $(b)\Rightarrow (a)$ note that we have already seen
that $[\![ T ]\!](u)$ and $[\![ T ]\!](v)$ are reduced.
Hence, $[\![ T ]\!](u) = [\![ T ]\!](v)$ in $G'$ together with
$\pi_t(u)=\pi_t(v)$ 
implies that there exists a Van Kampen diagram of the form $(\ddagger)$.
Again, we can replace the dark-shaded $t$-faces by the corresponding
$t$-faces of $G$ in order to obtain a diagram of the form $(\dagger)$.  
This proofs the claim.

\medskip

\noindent
By the previous claim, $[\![ T ]\!](\eval(\dA))$ and $[\![ T ]\!](\eval(\dB))$ are reduced.
Moreover, SLPs $\dA'$ and $\dB'$
for these strings can be computed in polynomial time by
Lemma~\ref{L transducer}. In case $\pi_t(\eval(\dA)) \neq \pi_t(\eval(\dB))$
we choose these SLPs such that e.g. $\eval(\dA')=t_1$ and 
$\eval(\dB')=t^{-1}_1$. Hence, $\eval(\dA') = \eval(\dB')$ in $G'$
if and only if 
$\eval(\dA) = \eval(\dB)$ in 
$\langle H, t_1, \ldots, t_n \mid a^{t_i} =\varphi_i(a) (1 \leq i \leq n, a
\in A) \rangle$.
This proves the lemma.
\qed
\end{proof}
Due to Lemma~\ref{L reducing to constant} it suffices to concentrate our effort on 
problems of the form $\RCWP(H,A,B,\varphi_1,\ldots, \varphi_k)$,
where $k\leq\delta$. Let
\begin{equation} \label{new G}
G_0 = \langle H, t_1, \ldots, t_k \mid a^{t_i} = \varphi_i(a)\ (1 \leq i
\leq k, a \in A) \rangle
\end{equation}
and let us choose $i \in \{1, \ldots, k\}$
such that $|\dom(\varphi_i)|$ is maximal. W.l.o.g. assume that $i=1$. 
Let $\dom(\varphi_1) = A_1 \leq A$ and 
$\ran(\varphi_1) = B_1 \leq B$.
We write $t$ for $t_1$ in the following and define $$\Gamma = \Sigma \cup \{t_2, \ldots, t_k\}.$$
We can write our HNN-extension $G_0$ from (\ref{new G}) as
\begin{equation} \label{HNN over K}
G_0 = \langle K, t \mid  a^ t = \varphi_1(a) \ (a \in A_1) \rangle,
\end{equation}
where
\begin{equation} \label{base K}
K = \langle H, t_2, \ldots, t_k \mid a^{t_i} = \varphi_i(a)\ (2 \leq i \leq k,
a \in A) \rangle.
\end{equation}
The latter group $K$ is generated by $\Gamma$. 
The goal of the next three Sections~\ref{S abstracting}--\ref{S transform into HNN} is to prove:

\begin{lemma} \label{L main}
$\RCWP(H,A,B,\varphi_1,\ldots, \varphi_k)$ is polynomial time Turing-reducible 
to the problems
$\RCWP(H,A,B,\varphi_2,\ldots,\varphi_k)$ and $\RUCWP(A_1,A_1,A_1)$.
\end{lemma}

\subsection{Abstracting from the base group $K$} \label{S abstracting}

Our aim in this subsection will be to reduce the compressed word problem for $G_0$ to
the compressed word problem for another group, where we have abstracted
from most of the concrete structure of the base group $K$ in (\ref{base K}).

Let us consider an input $(\dA,\dB)$ for $\RCWP(H,A,B,\varphi_1,\ldots,
\varphi_k)$ with $k \leq \delta$. W.l.o.g. assume that $k = \delta$.
Thus, $\dA$ and $\dB$ are SLPs over the alphabet
$\Sigma^{\pm 1} \cup \{t_1,t_1^{-1},\ldots, t_\delta,t_\delta^{-1}\} =
\Gamma^{\pm 1} \cup \{t,t^{-1}\}$
with $\eval(\dA), \eval(\dB) \in
\Red(H,\varphi_1,\ldots,\varphi_\delta)$.
Hence, we also have $\eval(\dA), \eval(\dB) \in \Red(K,\varphi_1)$.

W.l.o.g. we may assume that $\pi_t(\eval(\dA)) = \pi_t(\eval(\dB))$.
This property can be checked in polynomial time using Plandowski's
algorithm \cite{Pla94}, and if it is not satisfied then
we have $\eval(\dA) \neq \eval(\dB)$ in $G_0$.
Hence, there are $m \geq 0$,
$\alpha_1, \ldots, \alpha_m \in \{1,-1\}$, and
strings $u_0,v_0 \ldots, u_m,v_m \in (\Gamma^{\pm 1})^*$ such that
\begin{eqnarray}
\eval(\dA) & = & u_0 t^{\alpha_1} u_1  \cdots t^{\alpha_m} u_m  
\text{ and } \label{SLP A} \\
\eval(\dB) & = & v_0 t^{\alpha_1} v_1  \cdots t^{\alpha_m} v_m. \label{SLP B}
\end{eqnarray}
One might think that the number of different words $u_i$ (resp. $v_i$)
may grow exponentially in the size
of $\dA$ (resp. $\dB$). But we will see that this is actually not the case. 

Let us replace every occurrence of $t^\alpha$ ($\alpha\in \{1,-1\}$) in 
$\dA$ and $\dB$ by $a a ^{-1} t^\alpha a a^{-1}$, where $a \in \Gamma$ is arbitrary.
This is to ensure that any two occurrences of symbols from 
$\{t, t^{-1} \}$ are separated by 
a non-empty word over $\Gamma^{\pm 1}$, i.e., we can 
assume that $u_0,v_0,\ldots,u_m,v_m \in (\Gamma^{\pm 1})^+$
in (\ref{SLP A}) and (\ref{SLP B}).

Our first goal is to transform $\dA$ (and similarly $\dB$)
into an equivalent SLP that generates in a first phase
a string of the form $X_0 t^{\alpha_1} X_1  \cdots t^{\alpha_m} X_m$,
where $X_i$ is a further variable that generates in 
a second phase the string $u_i \in (\Gamma^{\pm 1})^+$. 
Assume that $\dA = (U, \{t,t^{-1}\} \cup \Gamma^{\pm1}, S,P)$ is in Chomsky
normal form.

In a first step, we remove every variable $X \in U$ from $\dA$ 
such that $X \to t$ or $X \to t^{-1}$ is a production of $\dA$ by
replacing $X$ in all right-hand sides of $\dA$ by $t$ or
$t^{-1}$, respectively.
Now, all productions of $\dA$ are of the form
$X \to YZ$, $X \to t^\alpha Z$, $X \to Y t^\alpha$, or
$X \to x \in \Gamma^{\pm 1}$, where $Y,Z \in U$.   

Next we split the set $U$ of variables of $\dA$ into two parts:
$$
U^0_K = \{ X \in U \mid \eval(X) \in (\Gamma^{\pm1})^+ \}
\qquad \text{and} \qquad
U^0_t = U \setminus U_K^0.
$$
Let $P^0_K$ (resp. $P^0_t$) 
be the set of all productions from $P$ with a left-hand
side in $U^0_K$ (resp. $U^0_t$).
The subscript $K$ refers to the fact that every nonterminal from $U^0_K$
defines an element from the new base group $K \leq G_0$, whereas the subscript 
$t$ refers to the fact that every nonterminal from $U^0_t$
generates a string where $K$-generators as well as $t$ or
$t^{-1}$ occurs.

Now we manipulate all productions from $P^0_t$ in a bottom-up
process, which adds further variables and productions to $U^0_K$ and
$P^0_K$, respectively. The set $U^0_t$ will not change in the process.
After stage $i$, we have production sets $P^i_t$ and 
$P^i_K$, and the set of left-hand sides of $P^i_t$
(resp. $P^i_K$) is $U^0_t$ (resp. $U^i_K$).
The system
$\dA^i_t := (U^0_t, \{t,t^{-1}\} \cup U^i_K, S, P^i_t)$
is a {\em composition system} that generates a string
from $(U^i_K)^+ t^{\alpha_1} (U^i_K)^+  \cdots t^{\alpha_m}
(U^i_K)^+$.

In stage $i+1$ we do the following: Consider a production
$(X \to u) \in P^i_t$ such that every variable in $u$
is already processed, but $X$ is not yet processed.
If $u$ is of the form $t^\alpha Z$ or $Y t^\alpha$, then there is nothing to do.
Now assume that $u = YZ$ such that
$Y$ and $Z$ are already processed.
Consider the last symbol 
$\omega \in \{t,t^{-1}\} \cup U^i_K$ of $\eval(\dA^i_t,Y)$ and the first 
symbol $\alpha \in \{t,t^{-1}\} \cup U^i_K$ of 
$\eval(\dA^i_t,Z)$ (these symbols can be computed
in polynomial time after stage $i$).
If either $\omega \in \{t,t^{-1}\}$ or $\alpha \in \{ t,t^{-1}\}$, 
then again nothing is to do.
Otherwise, $\omega, \alpha \in U^i_K$.
We now set $U^{i+1}_K = U^i_K \cup \{X'\}$, where $X'$ is 
a fresh variable, and $P^{i+1}_K = P^i_K \cup \{ X' \to \omega
\alpha\}$. Finally, we obtain $P^{i+1}_t$ from
$P^{i}_t$ by replacing the production 
$X \to YZ$ by $X \to Y[:\ell-1] X' Z[2:]$. Here
$\ell = |\eval(\dA^i_t,Y)|$. 

After the last stage, we transform the final composition system
$\dA^k_t$ (where $k$ is the number of stages)
into an equivalent SLP, let us denote this SLP by $\dA_t$.
Moreover, write $U_K$ and $P_K$ for 
$U^k_K$ and $P^k_K$.
The construction implies that
\begin{equation} \label{A_K,t}
\eval(\dA_t) = X_0 t^{\alpha_1} X_1  \cdots t^{\alpha_m} X_m
\end{equation}
with $X_0,\ldots,X_m \in U_K$ and 
$\eval(U_K, \Gamma^{\pm 1}, X_i, P_K) = u_i$.
Note that the number of different $X_i$ is polynomially bounded,
simply because the set $U_K$ was computed in polynomial time.
Hence, also the number of different $u_i$ in (\ref{SLP A})
is polynomially bounded.

For the SLP $\dB$ the same procedure yields the following data:
\begin{itemize}
\item An SLP $\dB_t$ such that
$$\eval(\dB_t) = Y_0 t^{\alpha_1} Y_1  \cdots t^{\alpha_m} Y_m.$$
\item A set of productions $Q_K$ with left-hand sides
$V_K$, where $\{Y_1,\ldots,Y_m\} \subseteq V_K$ and 
$\eval(V_K, \Gamma^{\pm 1}, Y_i, Q_K) = v_i$.
\end{itemize}
W.l.o.g. assume that $U_K \cap V_K = \emptyset$.
Let $W_K = U_K \cup V_K$ and 
$R_K = P_K \cup Q_K$.
In the following, for $Z \in W_K$ we write
$\eval(Z)$ for $\eval(W_K, \Gamma^{\pm 1}, Z, R_K) \in (\Gamma^{\pm 1})^+$.

Let us next consider the free product $F(\Wgamma) * A_1 * B_1$.
Recall that $A_1$ (resp. $B_1$) is the domain (resp. range)
of the partial isomorphism $\varphi_1$.
Clearly, in this free product, $A_1$ and $B_1$ have trivial
intersection (even if $A_1 \cap B_1 > 1$ in $H$). 
We now define a set of defining relations 
$\Emc$ by
\begin{multline} \label{rel E}
\Emc = \{ Z_1 c_1 = c_2 Z_2 \mid 
Z_1, Z_2 \in \Wgamma, c_1,c_2 \in A_1 \cup B_1, \\  
\eval(Z_1)\, c_1 = c_2\, \eval(Z_2) \text{ in the group } K \}.
\end{multline}
We can compute the set $\Emc$ in polynomial time using oracle access to
$\CWP(K)$ or alternatively, by Lemma~\ref{L reduced 2}, 
using oracle access to $\RCWP(H,A,B,\varphi_2,\ldots,\varphi_k)$. 
This is the only time, where we need oracle access 
to $\RCWP(H,A,B,\varphi_2,\ldots,\varphi_k)$ in Lemma~\ref{L main}.

Consider the group
\begin{eqnarray*}\label{G1}
G_1 & = &  \langle (F(\Wgamma) * A_1 * B_1)/N(\Emc), t\, \mid\, 
a^t = \varphi_1(a)\ (a \in A_1) \rangle \\
& = & \langle F(\Wgamma) * A_1 * B_1, t\, \mid\, \Emc, \,
t^{-1} a t = \varphi_1(a)\, (a \in A_1) \rangle .
\end{eqnarray*}
Recall that $N(\Emc) \leq F(\Wgamma) * A_1 * B_1$ is the smallest normal subgroup
of $F(\Wgamma) * A_1 * B_1$ that contains all elements $xy^{-1}$ with $(x=y) \in \Emc$.
We can define a morphism 
$$\psi : F(\Wgamma) * A_1 * B_1  \to K$$ by 
$\psi(Z) = \eval(Z)$
for  $Z \in \Wgamma$, $\psi(a)=a$ for $a \in A_1$, and $\psi(b)=b$ for $b \in
B_1$. Of course, the restrictions of $\psi$ to 
$A_1$ as well as $B_1$ are injective.
Moreover, each of the defining relations in $\Emc$ is preserved under $\psi$:
for $(Z_1 c_1 = c_2 Z_2) \in \Emc$ we have 
$\psi(Z_1 c_1) = \eval(Z_1)\, c_1 = c_2\, \eval(Z_2) = \psi(c_2 Z_2)$ 
in $K$. Thus, $\psi$ defines a morphism  
$$\widehat{\psi}: (F(\Wgamma) * A_1 * B_1)/N(\Emc) \to K.$$
Moreover, $A_1 \cap N(\Emc) = 1$: if $a \in N(\Emc) \cap A_1$ then $\psi(a) \in
\psi(N(\Emc)) = 1$; thus $a=1$, since $\psi$ is injective on $A_1$.
Similarly, $B_1 \cap N(\Emc) = 1$.  This means that
$A_1$ and $B_1$ can be naturally embedded in $(F(\Wgamma) * A_1 * B_1)/N(\Emc)$
and $\varphi_1 : A_1 \to B_1$ can be considered as an isomorphism between the 
images of this embedding in $(F(\Wgamma) * A_1 * B_1)/N(\Emc)$. 
Therefore, the group $G_1$ is an HNN-extension
with base group $(F(\Wgamma) * A_1 * B_1)/N(\Emc) \leq G_1$. Moreover,
$\widehat{\psi} : (F(\Wgamma) * A_1 * B_1)/N(\Emc) \to K$ 
can be lifted to a morphism 
$$\widehat{\psi} : G_1 \to G_0 = 
\langle K, t \mid  a^ t = \varphi_1(a) \ (a \in A_1) \rangle.$$
The idea for the construction of $G_1$
is to abstract as far as possible from the concrete structure
of the original base group $K$. We only keep those $K$-relations that
are necessary to prove (or disprove) that $\eval(\dA) = \eval(\dB)$ in the group $G_0$.

Note that since $\eval(\dA), \eval(\dB) \in \Red(K,\varphi_1)$, we have
$\eval(\dA_t), \eval(\dB_t) \in \Red((F(\Wgamma) * A_1 * B_1)/N(\Emc), \varphi_1)$:
Consider for instance a factor $t^{-1} X_i t$ of $\eval(\dA_t)$ from
(\ref{A_K,t}). If $X_i = a$ in $(F(\Wgamma) * A_1 * B_1)/N(\Emc)$ for some $a \in A_1$, then
after applying $\widehat{\psi}$ we have $\eval(X_i) = u_i = a$ in $K$. Hence,
$\eval(\dA)$ from (\ref{SLP A}) would not be reduced.

\begin{lemma} \label{Lemma K,t}
The following are equivalent:
\begin{enumerate}[(a)]
\item $\eval(\dA) = \eval(\dB)$ in $G_0$ from (\ref{HNN over K}).
\item $\eval(\dA_t) = \eval(\dB_t)$ in $G_1$
\end{enumerate}
\end{lemma}

\begin{proof}
For $(b)\Rightarrow (a)$ assume that $\eval(\dA_t) = \eval(\dB_t)$ in
$G_1$. We obtain in $G_0$:
$\eval(\dA) = \widehat{\psi}(\eval(\dA_t)) = \widehat{\psi}(\eval(\dB_t)) = \eval(\dB)$.

For $(a)\Rightarrow (b)$ assume that $\eval(\dA) = \eval(\dB)$ in the group
$G_0$. Since $\eval(\dA)$ and $\eval(\dB)$ are reduced and $\pi_t(\eval(\dA)) =
\pi_t(\eval(\dB))$, we obtain a Van Kampen diagram of the form:
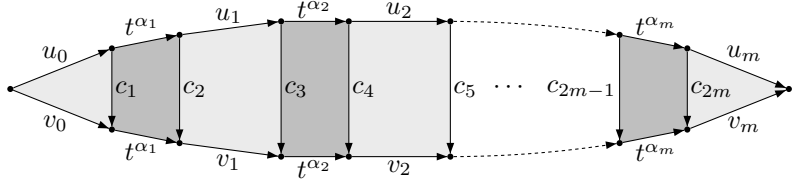
\begin{center}
\setlength{\unitlength}{.9mm}
\begin{picture}(115,26)(0,-3)
\gasset{Nframe=n,AHnb=1,ELdist=0.5}
\gasset{Nw=.9,Nh=.9,Nfill=y}
\drawpolygon[fillgray=0.92](0,10)(15,16)(15,4)
\drawpolygon[fillgray=0.75](15,4)(15,16)(25,18)(25,2)
\drawpolygon[fillgray=0.92](25,2)(25,18)(40,20)(40,0)
\drawpolygon[fillgray=0.75](40,0)(40,20)(50,20)(50,0)
\drawpolygon[fillgray=0.92](50,0)(50,20)(65,20)(65,0)
\drawpolygon[fillgray=1](65,0)(65,20)(75,20)(75,0)
\drawpolygon[fillgray=1](90,2)(90,18)(75,20)(75,0)
\drawpolygon[fillgray=0.75](100,4)(100,16)(90,18)(90,2)
\drawpolygon[fillgray=0.92](115,10)(100,16)(100,4)
\node(0)(0,10){}
\node(a)(15,16){}
\node(b)(25,18){}
\node(c)(40,20){}
\node(d)(50,20){}
\node(e)(65,20){}
\node(g)(90,18){}
\node(h)(100,16){}
\node(i)(115,10){}
\node(a')(15,4){}
\node(b')(25,2){}
\node(c')(40,0){}
\node(d')(50,0){}
\node(e')(65,0){}
\node(g')(90,2){}
\node(h')(100,4){}
\drawedge(0,a){$u_0$}
\drawedge(a,b){$t^{\alpha_1}$}
\drawedge(b,c){$u_1$}
\drawedge(c,d){$t^{\alpha_2}$}
\drawedge(d,e){$u_2$}
\drawedge(g,h){$t^{\alpha_m}$}
\drawedge(h,i){$u_m$}
\drawedge[ELside=r](0,a'){$v_0$}
\drawedge[ELside=r](a',b'){$t^{\alpha_1}$}
\drawedge[ELside=r](b',c'){$v_1$}
\drawedge[ELside=r](c',d'){$t^{\alpha_2}$}
\drawedge[ELside=r](d',e'){$v_2$}
\drawedge[curvedepth=.5,dash={.5}1](e,g){}
\drawedge[curvedepth=-.5,dash={.5}1](e',g'){}
\drawedge[ELside=r](g',h'){$t^{\alpha_m}$}
\drawedge[ELside=r](h',i){$v_m$}
\drawedge(a,a'){$c_1$}
\drawedge[ELpos=50](b,b'){$c_2$}
\drawedge[ELpos=50](c,c'){$c_3$}
\drawedge[ELpos=50](d,d'){$c_4$}
\drawedge[ELpos=50](e,e'){$c_5\;\, \cdots$}
\drawedge[ELpos=50,ELside=r](g,g'){$c_{2m-1}$}
\drawedge[ELpos=50](h,h'){$c_{2m}$}
\end{picture}
\end{center}
In this diagram, we can replace 
every light-shaded face, representing the $K$-relation
$u_i c_{2i+1}  = c_{2i} v_i$, by a face representing the valid
$\Emc$-relation $X_i c_{2i+1} = c_{2i} Y_i$, see (\ref{rel E}).
We obtain the following Van Kampen diagram, which shows
that $\eval(\dA_t) = \eval(\dB_t)$ in $G_1$:
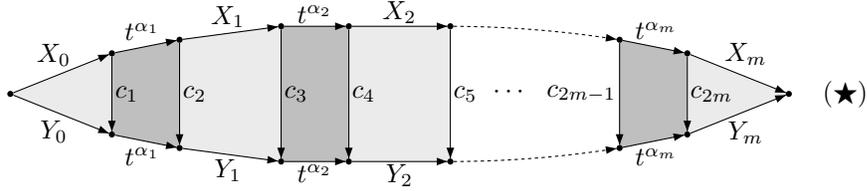
\begin{center}
\setlength{\unitlength}{.9mm}
\begin{picture}(115,26)(0,-3)
\gasset{Nframe=n,AHnb=1,ELdist=0.5}
\gasset{Nw=.9,Nh=.9,Nfill=y}
\put(120,9){$(\bigstar)$}
\drawpolygon[fillgray=0.92](0,10)(15,16)(15,4)
\drawpolygon[fillgray=0.75](15,4)(15,16)(25,18)(25,2)
\drawpolygon[fillgray=0.92](25,2)(25,18)(40,20)(40,0)
\drawpolygon[fillgray=0.75](40,0)(40,20)(50,20)(50,0)
\drawpolygon[fillgray=0.92](50,0)(50,20)(65,20)(65,0)
\drawpolygon[fillgray=1](65,0)(65,20)(75,20)(75,0)
\drawpolygon[fillgray=1](90,2)(90,18)(75,20)(75,0)
\drawpolygon[fillgray=0.75](100,4)(100,16)(90,18)(90,2)
\drawpolygon[fillgray=0.92](115,10)(100,16)(100,4)
\node(0)(0,10){}
\node(a)(15,16){}
\node(b)(25,18){}
\node(c)(40,20){}
\node(d)(50,20){}
\node(e)(65,20){}
\node(g)(90,18){}
\node(h)(100,16){}
\node(i)(115,10){}
\node(a')(15,4){}
\node(b')(25,2){}
\node(c')(40,0){}
\node(d')(50,0){}
\node(e')(65,0){}
\node(g')(90,2){}
\node(h')(100,4){}
\drawedge(0,a){$X_0$}
\drawedge(a,b){$t^{\alpha_1}$}
\drawedge(b,c){$X_1$}
\drawedge(c,d){$t^{\alpha_2}$}
\drawedge(d,e){$X_2$}
\drawedge(g,h){$t^{\alpha_m}$}
\drawedge(h,i){$X_m$}
\drawedge[ELside=r](0,a'){$Y_0$}
\drawedge[ELside=r](a',b'){$t^{\alpha_1}$}
\drawedge[ELside=r](b',c'){$Y_1$}
\drawedge[ELside=r](c',d'){$t^{\alpha_2}$}
\drawedge[ELside=r](d',e'){$Y_2$}
\drawedge[curvedepth=.5,dash={.5}1](e,g){}
\drawedge[curvedepth=-.5,dash={.5}1](e',g'){}
\drawedge[ELside=r](g',h'){$t^{\alpha_m}$}
\drawedge[ELside=r](h',i){$Y_m$}
\drawedge(a,a'){$c_1$}
\drawedge[ELpos=50](b,b'){$c_2$}
\drawedge[ELpos=50](c,c'){$c_3$}
\drawedge[ELpos=50](d,d'){$c_4$}
\drawedge[ELpos=50](e,e'){$c_5 \;\, \cdots$}
\drawedge[ELpos=50,ELside=r](g,g'){$c_{2m-1}$}
\drawedge[ELpos=50](h,h'){$c_{2m}$}
\end{picture}
\end{center}
\qed
\end{proof}
By Lemma~\ref{Lemma K,t}, it remains to check, whether 
$\eval(\dA_t) = \eval(\dB_t)$ in the HNN-extension $G_1$,
where $\eval(\dA_t)$ and $\eval(\dB_t)$ are both reduced.

\subsection{Eliminating $B_1$ and $t$}

By using the identities
$b = t^{-1} \varphi_1^{-1}(b) t$ ($b \in B_1 \setminus \{1\}$)
as Tietze transformations
we can eliminate in the group $G_1$
the generators from $B_1 \setminus \{1\}$.
After this transformation, we may have apart from relations of the form
\begin{equation}\label{def-rel-basic}
Z_1 a_1 = a_2 Z_2 \text{ with } a_1,a_2 \in A_1
\end{equation}
also defining relations of the forms
\begin{eqnarray*}
Z_1  t^{-1} a_1 t  & = & a_2 Z_2 \\
Z_1  a_1              & = & t^{-1} a_2 t Z_2 \\
Z_1  t^{-1} a_1 t  & = & t^{-1} a_2 t Z_2,
\end{eqnarray*}
where $a_1,a_2 \in A_1$.
We can replace these relations by relations of the following types
\begin{eqnarray}
Z_1  t^{-1} a_1   & = & a_2 Z_2 t^{-1} \label{z1}\\
t Z_1  a_1          & = & a_2 t Z_2 \label{z2}\\
t Z_1  t^{-1} a_1 & = &  a_2 t Z_2 t^{-1} \label{z3}
\end{eqnarray}
and end up with the isomorphic group
$$
G_2=\langle F(\Wgamma) * A ,t \mid (\ref{def-rel-basic})-(\ref{z3}) \rangle.
$$
Let us now introduce for every $Z \in \Wgamma$ the new generators
$$
[Z t^{-1}], [tZ], [t Z t^{-1}]
$$
together with the defining relations
\begin{equation}
[Z t^{-1}] = Z t^{-1},   \ [tZ] = tZ, \  [t Z t^{-1}] = t Z t^{-1}. \label{rel-new}
\end{equation}
\noindent
This allows to replace the defining relations 
(\ref{z1})--(\ref{z3}) by
\begin{eqnarray}
[Z_1  t^{-1}] a_1   & = & a_2 [Z_2 t^{-1}]    \label{z1-new} \\
\text{$[ t Z_1]  a_1$}          & = & a_2 [ t Z_2]            \label{z2-new} \\
\text{$[ t Z_1 t^{-1}] a_1$} & = &  a_2 [ t Z_2 t^{-1}] \label{z3-new}
\end{eqnarray}
\noindent
leading to the group 
\begin{equation}
G_3=\langle F(\{Z,[Zt^{-1}],[tZ],[tZt^{-1}]|Z\in\Wgamma\}) * A_1 ,t \mid 
(\ref{def-rel-basic}),(\ref{rel-new})-(\ref{z3-new})\rangle.
\end{equation}
Finally, we can eliminate $t$ and $t^{-1}$ by replacing (\ref{rel-new}) by 
\begin{equation}
 [tZ] = [Z t^{-1}]^{-1} Z^2, \  [t Z t^{-1}] = [t Z] Z^{-1}  [Z t^{-1}]. \label{rel-new-new}
\end{equation}
Doing this replacement we end up with the group
\begin{equation} \label{G4}
G_4=\langle F(\{Z, [Z t^{-1}], [tZ], [t Z t^{-1}]\mid Z \in \Wgamma\}) * A_1 \mid  \text{(\ref{def-rel-basic}), 
(\ref{z1-new})-(\ref{z3-new}), (\ref{rel-new-new})}  \rangle .
\end{equation}
Since each transformation from $G_1$ to $G_4$ is a Tietze
transformation, $G_1$ is isomorphic to $G_4$.
We now want to rewrite the SLPs $\dA_t$ and $\dB_t$ into new SLPs over the 
generators of $G_4$. For this, we can define a deterministic 
rational transducer $T$ that reads a word
$X_0 t^{\alpha_1} X_1 t^{\alpha_2} X_2 \cdots t^{\alpha_m} X_m$ from the
input tape and
\begin{itemize}
\item replaces every occurrence of a factor $t X_i$ with $\alpha_{i+1} \neq -1$ by
the symbol $[tX_i]$,
\item replaces every occurrence of a factor $X_i t^{-1}$ with $\alpha_i \neq 1$ by the symbol $[X_i t^{-1}]$, and finally
\item replaces every occurrence of a factor $t X_i t^{-1}$ by the symbol $[tX_i t^{-1}]$.
\end{itemize}
The state set of the transducer $T$ is $\{ \varepsilon, t\} \cup \{Z, tZ \mid
Z \in W_K\}$ and the transitions are the following (for all $Z,Z' \in W_k$),
where $\$$ is an end marker:
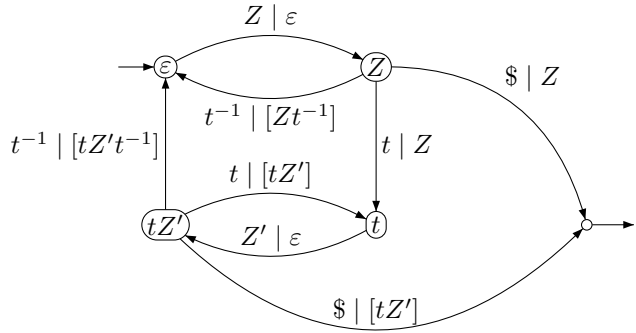
\begin{center}
\setlength{\unitlength}{1.4mm}
\begin{picture}(65,40)(-10,-30)
\gasset{Nframe=y,AHnb=1,ELdist=0.5}
\gasset{Nw=2,Nh=2,Nfill=n,Nadjust=wh,Nadjustdist=.5}
\node(a)(0,0){$\varepsilon$}
\node(b)(20,0){$Z$}
\node(d)(0,-15){$t Z'$}
\node(e)(20,-15){$t$}
\node[Nframe=n](h)(-5,0){}
\node(i)(40,-15){}
\node[Nframe=n](j)(45,-15){}
\drawedge(h,a){}
\drawedge[curvedepth=3](a,b){$Z \mid \varepsilon$}
\drawedge[curvedepth=3](b,a){$t^{-1}  \mid  [Zt^{-1}]$}
\drawedge(b,e){$t \mid Z$}
\drawedge[curvedepth=5](b,i){$\$  \mid  Z$}
\drawedge(i,j){}
\drawedge(d,a){$t^{-1}  \mid  [t Z' t^{-1}]$}
\drawedge[ELside=l,curvedepth=-10](d,i){$\$  \mid  [tZ']$}
\drawedge[curvedepth=3,ELside=r](e,d){$Z'  \mid  \varepsilon$}
\drawedge[curvedepth=3](d,e){$t \mid [t Z']$}
\end{picture}
\end{center}
\noindent
By Lemma \ref{L transducer} we can construct in polynomial time SLPs 
that generate the strings $[\![T]\!](\eval(\dA_t)\$)$ and $[\![T]\!](\eval(\dB_t)\$)$.

Let $G_5$ be the group that is obtained by removing the relations
(\ref{rel-new-new}) from the presentation of $G_4$ in (\ref{G4}), i.e.,
\begin{equation} \label{nn-gen-pres}
G_5 = \langle F(\{Z, [Z t^{-1}], [tZ], [t Z t^{-1}]\mid Z \in W_K\}) * A_1 \mid
\text{(\ref{def-rel-basic}), (\ref{z1-new})--(\ref{z3-new})} \rangle .
\end{equation}

\begin{lemma}
The following are equivalent:
\begin{enumerate}[(a)]
\item $\eval(\dA) = \eval(\dB)$ in $G_0$
\item $\eval(\dA_t) = \eval(\dB_t)$ in $G_1$
\item $[\![T]\!](\eval(\dA_t)\$) = [\![T]\!](\eval(\dB_t)\$)$  in $G_4$ 
\item $[\![T]\!](\eval(\dA_t)\$) = [\![T]\!](\eval(\dB_t)\$)$  in $G_5$
\end{enumerate}
\end{lemma}

\begin{proof}
The equivalence of (a) and (b) was stated in Lemma~\ref{Lemma K,t}.
The equivalence of (b) and (c) is clear since $G_1$ and $G_4$ are isomorphic
and the transducer $T$ rewrites a string over the generators $G_1$ into
a string over the generators of $G_4$.
Moreover, (d) implies (c) because we omit  one type of relations, namely (\ref{rel-new-new}),
when going from $G_5$ to $G_4$.
It remains to prove that (a) implies (d). 
If $\eval(\dA) = \eval(\dB)$ in $G_0$, then, as argued in the proof
of Lemma~\ref{Lemma K,t}, we obtain a Van Kampen diagram of the form 
$(\bigstar)$ in the group $G_1$. The boundary of every light-shaded face is
labeled with a relation from ${\cal E}$.
We obtain a Van Kampen diagram for 
$[\![T]\!](\eval(\dA_t)\$) = [\![T]\!](\eval(\dB_t)\$)$ in $G_5$, basically
by removing all vertical edges that connect (i) target nodes of 
$t$-labeled edges or (ii) source nodes of $t^{-1}$-labeled edges
(there are $B_1$-labeled edges in $(\bigstar)$),
see the following example. 
\qed
\end{proof}

\begin{example}
Let us give an example of the transformation from a diagram of 
the form $(\bigstar)$ into a Van Kampen diagram for the group $G_5$.
Assume that the diagram in $G_1$ is:
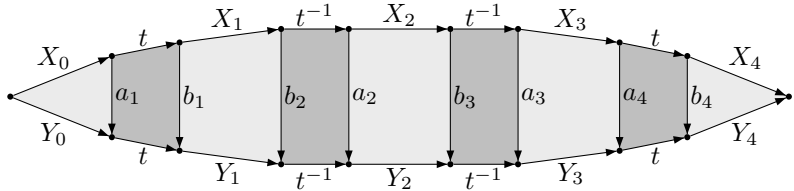
\begin{center}
\setlength{\unitlength}{.9mm}
\begin{picture}(115,26)(0,-3)
\gasset{Nframe=n,AHnb=1,ELdist=0.5}
\gasset{Nw=.9,Nh=.9,Nfill=y}
\drawpolygon[fillgray=0.92](0,10)(15,16)(15,4)
\drawpolygon[fillgray=0.75](15,4)(15,16)(25,18)(25,2)
\drawpolygon[fillgray=0.92](25,2)(25,18)(40,20)(40,0)
\drawpolygon[fillgray=0.75](40,0)(40,20)(50,20)(50,0)
\drawpolygon[fillgray=0.92](50,0)(50,20)(65,20)(65,0)
\drawpolygon[fillgray=0.75](65,0)(65,20)(75,20)(75,0)
\drawpolygon[fillgray=0.92](90,2)(90,18)(75,20)(75,0)
\drawpolygon[fillgray=0.75](100,4)(100,16)(90,18)(90,2)
\drawpolygon[fillgray=0.92](115,10)(100,16)(100,4)
\node(0)(0,10){}
\node(a)(15,16){}
\node(b)(25,18){}
\node(c)(40,20){}
\node(d)(50,20){}
\node(e)(65,20){}
\node(f)(75,20){}
\node(g)(90,18){}
\node(h)(100,16){}
\node(i)(115,10){}
\node(a')(15,4){}
\node(b')(25,2){}
\node(c')(40,0){}
\node(d')(50,0){}
\node(e')(65,0){}
\node(f')(75,0){}
\node(g')(90,2){}
\node(h')(100,4){}
\drawedge(0,a){$X_0$}
\drawedge(a,b){$t$}
\drawedge(b,c){$X_1$}
\drawedge(c,d){$t^{-1}$}
\drawedge(d,e){$X_2$}
\drawedge(e,f){$t^{-1}$}
\drawedge(f,g){$X_3$}
\drawedge(g,h){$t$}
\drawedge(h,i){$X_4$}
\drawedge[ELside=r](0,a'){$Y_0$}
\drawedge[ELside=r](a',b'){$t$}
\drawedge[ELside=r](b',c'){$Y_1$}
\drawedge[ELside=r](c',d'){$t^{-1}$}
\drawedge[ELside=r](d',e'){$Y_2$}
\drawedge[ELside=r](e',f'){$t^{-1}$}
\drawedge[ELside=r](f',g'){$Y_3$}
\drawedge[ELside=r](g',h'){$t$}
\drawedge[ELside=r](h',i){$Y_4$}
\drawedge(a,a'){$a_1$}
\drawedge[ELpos=50](b,b'){$b_1$}
\drawedge[ELpos=50](c,c'){$b_2$}
\drawedge[ELpos=50](d,d'){$a_2$}
\drawedge[ELpos=50](e,e'){$b_3$}
\drawedge[ELpos=50](f,f'){$a_3$}
\drawedge[ELpos=50](g,g'){$a_4$}
\drawedge[ELpos=50](h,h'){$b_4$}
\end{picture}
\end{center}
Then we obtain the following Van Kampen diagram in the group $G_5$:
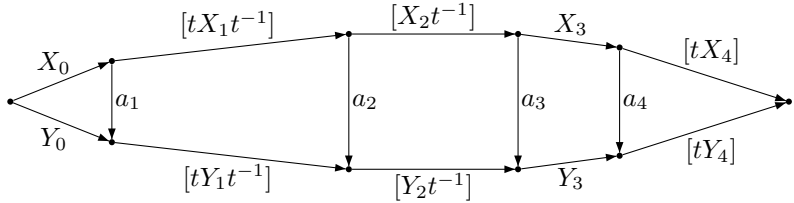
\begin{center}
\setlength{\unitlength}{.9mm}
\begin{picture}(115,26)(0,-3)
\gasset{Nframe=n,AHnb=1,ELdist=0.5}
\gasset{Nw=.9,Nh=.9,Nfill=y}
\drawpolygon[fillgray=1](0,10)(15,16)(15,4)
\drawpolygon[fillgray=1](15,4)(15,16)(25,18)(25,2)
\drawpolygon[fillgray=1](25,2)(25,18)(40,20)(40,0)
\drawpolygon[fillgray=1](40,0)(40,20)(50,20)(50,0)
\drawpolygon[fillgray=1](50,0)(50,20)(65,20)(65,0)
\drawpolygon[fillgray=1](65,0)(65,20)(75,20)(75,0)
\drawpolygon[fillgray=1](90,2)(90,18)(75,20)(75,0)
\drawpolygon[fillgray=1](100,4)(100,16)(90,18)(90,2)
\drawpolygon[fillgray=1](115,10)(100,16)(100,4)
\node(0)(0,10){}
\node(a)(15,16){}
\node(d)(50,20){}
\node(f)(75,20){}
\node(g)(90,18){}
\node(i)(115,10){}
\node(a')(15,4){}
\node(d')(50,0){}
\node(f')(75,0){}
\node(g')(90,2){}
\drawedge(0,a){$X_0$}
\drawedge(a,d){$[t X_1 t^{-1}]$}
\drawedge(d,f){$[X_2 t^{-1}]$}
\drawedge(f,g){$X_3$}
\drawedge(g,i){$[t X_4]$}
\drawedge[ELside=r](0,a'){$Y_0$}
\drawedge[ELside=r](a',d'){$[t Y_1 t^{-1}]$}
\drawedge[ELside=r](d',f'){$[Y_2 t^{-1}]$}
\drawedge[ELside=r](f',g'){$Y_3$}
\drawedge[ELside=r](g',i){$[t Y_4]$}
\drawedge(a,a'){$a_1$}
\drawedge[ELpos=50](d,d'){$a_2$}
\drawedge[ELpos=50](f,f'){$a_3$}
\drawedge[ELpos=50](g,g'){$a_4$}
\end{picture}
\end{center}
Only the relations (\ref{def-rel-basic}) and (\ref{z1-new})--(\ref{z3-new}) are used in this diagram.
\end{example}
For the further considerations, we denote the SLPs for
the strings $[\![T]\!](\eval(\dA_t)\$)$ and $[\![T]\!](\eval(\dB_t)\$)$ again with 
$\dA$ and $\dB$, respectively.
It remains to check whether $\eval(\dA) = \eval(\dB)$ in $G_5$.
Let 
$$\Zmc = \{Z, [Z t^{-1}], [tZ], [t Z t^{-1}]\mid Z \in \Wgamma\}$$ 
and let us redefine the set of defining relations $\Emc$ as
the set of all defining relations
of the form (\ref{def-rel-basic}), (\ref{z1-new})--(\ref{z3-new}).
Thus, 
$$
G_5=\langle F(\Zmc) * A_1 \mid  \Emc \rangle,
$$ 
where every defining relation in $\Emc$ is of the form
$Z_1 a_1 = a_2 Z_2$ for $Z_1, Z_2 \in \Zmc$ and $a_1,a_2 \in A_1$.

\subsection{Transforming  $\langle F(\Zmc) * A_1 \mid  \Emc \rangle$ into an
  HNN-extension} \label{S transform into HNN}

By further Tietze transformations we will show that $G_5$ is actually 
an HNN-extension with base group $A_1$ and associated subgroups 
$A_1$ and $A_1$. This will prove Lemma~\ref{L main}.
To this end, let us take a relation $Z_1 a_1 = a_2 Z_2$
with $Z_1 \neq Z_2$. We can eliminate $Z_2$ by replacing 
it with $a_2^{-1} Z_1 a_1$. Subwords of the form $a a'$ with $a,a' \in A_1$
that arise after this Tietze transformation can of course be multiplied out
in the finite group $A_1$. We carry out the same replacement $Z_2 \mapsto 
a_2^{-1} Z_1 a_1$ also in the SLPs  $\dA$ and $\dB$ which increases the size
only by an additive constant and repeat these steps. 
After polynomially many Tietze transformations we arrive at
a presentation, where all defining relations are of the form $Z = a_1 Z a_2$,
i.e. $a_2 = Z^{-1} a_1^{-1} Z$.
Let us write the resulting presentation as 
$$
G_6 =\langle A_1, Z_1, \ldots, Z_m \mid Z_i^{-1} a Z_i = \psi_i(a) \
(1 \leq i \leq m, a \in \dom(\psi_i)) \rangle.
$$
Note that every mapping $\psi_i$ is a partial automorphism on $A_1$ since
it results from the conjugation by some element in our initial group.
Hence, we obtained an HNN-extension over $A_1$.

We can now finish the proof of Lemma~\ref{L main}, which
states that the problem $\RCWP(H,A,B,\varphi_1,\ldots, \varphi_k)$ is
polynomial time Turing-reducible to the problems
$\RCWP(H,A,B,\varphi_2,\ldots,\varphi_k)$ and $\RUCWP(A_1,A_1,A_1)$. 
Using oracle access 
to $\RCWP(H,A,B,\varphi_2,\ldots, \varphi_k)$ 
(which was necessary for computing the set of defining relations
$\Emc$ from (\ref{rel E})), we have computed in polynomial time 
from a given $\RCWP(H,A,B,\varphi_1,\ldots, \varphi_k)$-instance an 
$\UCWP(A_1,A_1,A_1)$-instance, which is a positive instance
if and only if the original 
$\RCWP(H,A,B,\varphi_1,\ldots, \varphi_k)$-instance is positive.
A final application of Lemma~\ref{L reduced}
allows to reduce $\UCWP(A_1,A_1,A_1)$ to $\RUCWP(A_1,A_1,A_1)$.
This finishes the proof of Lemma~\ref{L main}.

\subsection{Finishing the proof of Theorem~\ref{T real main}}

We now apply Lemma~\ref{L reducing to constant}
to the problem $\RUCWP(A_1,A_1,A_1)$ (one of the two target problems
in Lemma~\ref{L main}). An input for this problem can 
be reduced in polynomial time to an instance of a problem
$\RCWP(A_1,A_1,A_1,\psi_1,\ldots,\psi_k)$, where
$\psi_1,\ldots,\psi_k : A_1 \to A_1$ and $k \leq \delta$
(we even have $k \leq 2 |A_1|! \cdot 2^{|A_1|} \leq 2 |A|! \cdot 2^{|A|} = \delta$).

We now separate the (constantly many) stable letters $t_1, \ldots, t_k$ 
that occur in the $\RCWP(A_1,A_1,A_1,\psi_1,\ldots,\psi_k)$-instance
into two sets:
$\{ t_1, \ldots, t_k \} = S_1 \cup S_2$ where $S_1 = \{ t_i \mid
\dom(\psi_i) =  A_1\}$ and 
$S_2 = \{ t_1, \ldots, t_k \} \setminus S_1$. W.l.o.g. assume that
$S_2 = \{t_1, \ldots, t_\ell \}$.
Then we can write our HNN-extension $G_6$ as
\begin{equation}  \label{G6}
G_6 = \langle H', t_1, \ldots, t_\ell \mid   a^ {t_i} = \psi_i(a) \ (1 \leq i
\leq \ell, a \in \dom(\psi_i) \rangle,
\end{equation}
where
$$
H' = \langle A_1, t_{\ell+1}, \ldots, t_k \mid  a^{t_i} = \psi_i(a) \ (\ell+1 \leq i \leq k, a \in A_1) \rangle.
$$
Note that $|\dom(\psi_i)| < |A_1|$ for every $1 \leq i \leq \ell$
and that $A_1 = \dom(\psi_i)$ for every $\ell+1 \leq i \leq k$. 
By Lemma~\ref{cwp semidirect}, 
$\CWP(H')$ can be solved in polynomial time; 
$H'$ is in fact the semidirect product $A_1 \rtimes_{\varphi} F(t_{\ell+1},
\ldots, t_k)$, where $\varphi : F(t_{\ell+1},
\ldots, t_k) \to \text{Aut}(A_1)$ is defined by
$\varphi(t_i) = \psi_i$. 
Recall also that at the end of Section~\ref{S constant number},
$A_1$ was chosen to be of maximal cardinality among the domains
of all partial isomorphisms $\varphi_1,\ldots,\varphi_k$.
The following proposition summarizes what we have shown so far:

\begin{proposition}\label{prop-reduce}
Let $\varphi_1, \ldots, \varphi_k : A \to B$
be partial isomorphisms, where 
$k \leq \delta$, $A_1 = \dom(\varphi_1)$,
and w.l.o.g  $|A_1| \geq |\dom(\varphi_i)|$
for $1 \leq i \leq k$.
From an instance $(\dA,\dB)$ of the problem 
$\RCWP(H,A,B,\varphi_1,\ldots,\varphi_k)$ 
we can compute in polynomial time with
oracle access to the problem 
$\RCWP(H,A,B,\varphi_2,\ldots,\varphi_k)$
\begin{enumerate}[(1)]
\item a semidirect product 
$A_1 \rtimes_{\varphi} F$, where $F$ is a free group
of rank at most $\delta$, 
\item partial automorphisms $\psi_1,\ldots,\psi_\ell : A_1 \to A_1$
with $\ell \leq \delta$ and 
$|\dom(\psi_i)| < |A_1|$ for all $1 \leq i \leq \ell$, and
\item an  
$\RCWP(A_1 \rtimes_{\varphi} F, A_1, A_1, \psi_1,\ldots,\psi_\ell)$-instance,
which is positive if and only if the initial 
$\RCWP(H,A,B,\varphi_1,\ldots,\varphi_k)$-instance $(\dA,\dB)$
is positive.
\end{enumerate}
\end{proposition}
Note that in (1) there are only
constantly many semidirect products of the form
$A_1 \rtimes_{\varphi} F$
and that $\CWP(A_1 \rtimes_{\varphi} F)$ can be solved
in polynomial time by Lemma~\ref{cwp semidirect}.

We are now ready to prove the main theorem of this paper.

\medskip

\noindent
{\em Proof of Theorem~\ref{T real main}.}
By Lemma~\ref{L reduced} and Lemma~\ref{L reducing to constant}
it suffices to solve a problem
$\RCWP(H,A,B,\varphi_1,\ldots,\varphi_k)$ (with $k \leq \delta$) in polynomial time.
For this we apply Proposition~\ref{prop-reduce} repeatedly.
We obtain a computation tree, where the root is labeled with 
an $\RCWP(H,A,B,\varphi_1,\ldots,\varphi_k)$-instance and
every other node is labeled with an instance of a problem
$\RCWP(C \rtimes_{\varphi} F, C, C, \theta_1, \ldots, \theta_p)$, where
$F$ is a free group of rank at most $\delta$, $C$ is a subgroup of 
our finite group $A$,
and $p \leq \delta$. The number of these problems is bounded by
some fixed constant. 
Since along each edge in the tree, either the number of stable letters reduces
by one, or the maximal size of an associated subgroup becomes strictly smaller,
the height of the tree is bounded by a constant (it is at most
$|A| \cdot \delta = 2 \cdot |A| \cdot |A|! \cdot 2^{|A|}$).
Moreover, along each tree edge, the size of a problem instance can 
grow only polynomially. Hence, each problem instance that appears
in the computation tree has polynomial size w.r.t. the input size.
Hence, the total running time is bounded polynomially.
\qed

\section{Amalgamated Products}\label{amalprod}

In this section we prove a transfer theorem for 
the compressed word problem for an amalgamated free product,
where the amalgamated subgroups are finite. We will deduce
this result from our transfer theorem for HNN-extensions.

Let $H_1$ and $H_2$ be two finitely generated groups. 
Let $A_1\leq H_1$ and $A_2\leq H_2$ be finite and
$\varphi : A_1\mapsto A_2$ an isomorphism. The 
\textit{amalgamated free product of $H_1$ and $H_2$, amalgamating the 
subgroups $A_1$ and $A_2$ by the isomorphism $\varphi$}, is the group
$$G=\langle H_1*H_2 \mid a=\varphi(a)\ (a\in A_1) \rangle.$$

\begin{theorem}\label{T_main_amal}
Let $G = \langle H_1*H_2 \mid a=\varphi(a) \ (a\in A_1)  \rangle$ be an 
amalgamated free product with $A_1$ finite. Then
$\CWP(G) \leq_T^P \{\CWP(H_1),\CWP(H_2)\}$.
\end{theorem}

\begin{proof}
It is well known \cite[Theorem 2.6, p. 187]{LySch77}
that $G$ can be embedded into the HNN-extension 
$$G':= \langle H_1*H_2,t \mid a^t=\varphi(a)\ (a\in A_1)\rangle$$
by the homomorphism $\Phi$ with
$$\Phi(x) \ = \ 
    \begin{cases} t^{-1}xt & \text{ if } x \in H_1 \\
                  x        & \text{ if } x \in H_2.
    \end {cases}
$$
Given an SLP $\dA$ we can easily compute an SLP $\dB$ with
$\eval(\dB) = \Phi(\eval(\dA))$. We obtain
\begin{eqnarray*}
\eval(\dA)=1 \text{ in } G  & \iff & \Phi(\eval(\dA)) =1 \text{ in } \Phi(G)\\
                        & \iff & \eval(\dB)=1 \text{ in } G'. 
\end{eqnarray*}
By Theorem \ref{T real main} and Theorem \ref{T free},
$\CWP(G')$ can be solved in polynomial
time with oracle access to $\CWP(H_1)$ and $\CWP(H_2)$.
\qed
\end{proof}

\section{Open Problems}\label{openprobs}

We have shown that the compressed word problem for an HNN-extension with
finite associated subgroups is polynomial time Turing-reducible to 
the compressed word problem for the base group. Here, the base group and 
the associated subgroups are fixed, i.e. are not part of the input.
One might also consider the \emph{uniform} compressed word problem for HNN-extensions
of the form $\langle H, t \mid a^t = \varphi(a) \ (a \in A)\rangle$,
where $H$ is a finite group that is part of the input. It is not clear,
whether this problem can be solved in polynomial time.

One might also consider the compressed word problem for HNN-extensions of 
semigroups \cite{How63}.

\bibliographystyle{abbrv}

\def\cprime{$'$}

\end{document}